\newcommand{\si}[1]{#1}
\newcommand{\jo}[1]{}

\si{
\documentclass[10pt,a4paper]{article}
\usepackage[utf8]{inputenc}
\usepackage{graphicx}
\usepackage{amsfonts}
\usepackage{amsthm}
\usepackage{amsmath}
\usepackage{amssymb}
\usepackage{array}
\usepackage{xcolor}
\usepackage{algorithm}
\usepackage{indentfirst}
\usepackage{algorithmic}
\usepackage{chngcntr}
\usepackage[left=2cm,right=2cm,top=2cm,bottom=2cm]{geometry}

\tracingstats=0

\newtheorem{theorem}{Theorem}[section]
\newtheorem{proposition}{Proposition}[section]
\newtheorem{corollary}{Corollary}[section]
\newtheorem{lemma}{Lemma}[section]
\newtheorem{definition}{Definition}[section]

\hypersetup{
  colorlinks=true,
  citecolor=blue,
  linkcolor=blue,
  filecolor=magenta,      
  urlcolor=cyan,
}
}

\jo{
\RequirePackage{fix-cm}
\documentclass[smallextended]{svjour3}       
\smartqed  

\usepackage[utf8]{inputenc}
\usepackage{graphicx}
\usepackage{amsfonts}
\usepackage{amsmath}
\usepackage{amssymb}
\usepackage{array}
\usepackage{xcolor}
\usepackage{algorithm}
\usepackage{algorithmic}
\usepackage{chngcntr}

\tracingstats=0
}

\newcommand{\E}{\mathbb{E}}						 
\newcommand{\pcon}[1]{\Pi_{\K}(#1)}				 

\newcommand{\seq}[1]{\{#1^k\}_{k\in \N}}		 
\renewcommand{\bar}{\overline}                   
\newcommand{\bd}{\mathrm{bd}}                    
\newcommand{\grad}{\nabla}                       
\newcommand{\I}{\mathbb{I}}                      
\newcommand{\inner}[2]{\langle#1,#2\rangle}      
\newcommand{\interior}{\mathrm{int}}             
\newcommand{\K}{\mathcal{K}}                     
\newcommand{\lin}{\mathrm{lin}}                  
\newcommand{\norm}[1]{\|#1\|}                    
\newcommand{\enorm}[1]{\|#1\|_2}                 
\newcommand{\fnorm}[1]{\|#1\|_F}                 
\newcommand{\projs}{\Pi_{\mathbb{S}^m_+}}        
\newcommand{\R}{\mathbb{R}}  					 
\newcommand{\Lor}{\mathbb{L}}                    
\renewcommand{\Re}{\mathbb{R}}                   
\newcommand{\T}{\top\hspace{-1pt}}               
\newcommand{\N}{\mathbb{N}}                      
\newcommand{\s}{^{\star}}                        
\newcommand{\pol}{^{\textnormal{o}}}             
\newcommand{\sym}{\mathbb{S}^m}                  

\makeatletter
\newcommand{\tpitchfork}{%
  \vbox{
    \baselineskip\z@skip
    \lineskip-.52ex
    \lineskiplimit\maxdimen
    \m@th
    \ialign{##\crcr\hidewidth\smash{$-$}\hidewidth\crcr$\pitchfork$\crcr}
  }%
}
\makeatother                                     

\si{
}

\jo{
\usepackage[misc]{ifsym}
\newcommand{\corr}{$^{\textrm{\Letter}}$}
}

\begin{document}

\title{On the weak second-order optimality condition for nonlinear semidefinite and second-order cone programming\footnotetext{The authors received financial support from FAPESP (grants 2018/24293-0, 2017/18308-2, and 2017/17840-2), CNPq (grants 303427/2018-3 and 404656/2018-8), PRONEX - CNPq/FAPERJ (grant E-26/010.001247/2016), and a Grant-in-Aid for scientific research (C)(19K11840) from Japan Society for the Promotion of Science.}}

\jo{\titlerunning{Second-order analysis for NSDP and NSOCP via sequential conditions}}

\si{
\author{
Ellen H. Fukuda \thanks{Graduate School of Informatics,
    Kyoto University, Kyoto, Japan.
	Email: {\tt ellen@i.kyoto-u.ac.jp}}	
\and
Gabriel Haeser \thanks{Department of Applied Mathematics, University of S{\~a}o Paulo, S{\~a}o Paulo, SP, Brazil. 
	Email: {\tt ghaeser@ime.usp.br}}	
\and 
Leonardo M. Mito \thanks{Department of Applied Mathematics, University of S{\~a}o Paulo, S{\~a}o Paulo, SP, Brazil. 
	Email: {\tt leokoto@ime.usp.br}}
}
}

\jo{
\author{
Ellen H. Fukuda
\and
Gabriel Haeser
\and 
Leonardo M. Mito
}

\authorrunning{E. H. Fukuda, G. Haeser, and L. M. Mito} 

\institute{ 
    Ellen H. Fukuda \at
    Department of Applied Mathematics and Physics,
    Kyoto University, Kyoto, Japan.\\
    \email{ellen@i.kyoto-u.ac.jp}
    \and
    Gabriel Haeser \at
    Department of Applied Mathematics, 
    University of S\~ao Paulo, 
    S{\~a}o Paulo, Brazil.\\
    \email{ghaeser@ime.usp.br} 
    \and
    Leonardo M. Mito \corr 
    \at
    Department of Applied Mathematics, 
    University of S{\~a}o Paulo, 
    S{\~a}o Paulo, Brazil.\\
	\email{leokoto@ime.usp.br}	
  }
}

\jo{\date{Received: date / Accepted: date}}

\maketitle

\abstract{Second-order necessary optimality conditions for nonlinear conic programming problems that depend on a single Lagrange multiplier are usually built under nondegeneracy and strict complementarity. In this paper we establish a condition of such type for two classes of nonlinear conic problems, namely semidefinite and second-order cone programming, assuming Robinson’s constraint qualification and a weak constant rank-type property which are, together, strictly weaker than nondegeneracy. Our approach is done via a penalty-based strategy, which is aimed at providing strong global convergence results for first- and second-order algorithms. Since we are not assuming strict complementarity, the critical cone does not reduce to a subspace, thus, the second-order condition we arrive at is defined in terms of the lineality space of the critical cone. In the case of nonlinear programming, this condition reduces to the standard second-order condition widely used as second-order stationarity measure in the algorithmic practice.

\si{
\

\textbf{Keywords:} Optimality conditions, Semidefinite programming, Second-order cone programming.}
}

\jo{\keywords{Optimality conditions \and Semidefinite programming \and Second-order cone programming}

\subclass{ 90C46 \and 90C30 \and 90C26 \and 90C22}
}


\section{Introduction}

Consider the following \textit{nonlinear conic programming} (NCP) problem in standard form:

\begin{equation}
  \tag{NCP}
  \begin{aligned}
    & \underset{x \in \mathbb{R}^{n}}{\text{Minimize}}
    & & f(x), \\
    & \text{subject to}
    & & g(x) \in  \K, \\
    & & & h(x) = 0,
  \end{aligned}
  \label{NCP}
\end{equation}
where $f\colon\R^{n}\to\R$, $g\colon\R^n \to \E$ and $h \colon \R^n
\to \R^p$ are twice continuously differentiable functions, $\E$ is a finite-dimensional linear space equipped with an inner product $\inner{\cdot}{\cdot}$ and the norm $\norm{\cdot}$ induced by it, and $\mathcal{K}\subseteq \E$ is a closed convex cone that is assumed to be \textit{self-dual}, which means $\K=\K^*\doteq \{w\in \E\colon \forall y\in \K, \langle w, y\rangle\geqslant 0\}$.

We are primarily interested in second-order necessary optimality conditions for two well-established particular cases of \eqref{NCP}:

\begin{itemize}
\item \hyperref[sec:socp]{\textit{Nonlinear second-order cone programming}} (NSOCP), which is obtained when $\E=\R^m$ and $\K$ is the so-called (Lorentz) \textit{second-order cone}, defined as $\Lor^{m}\doteq\{(w_0,\bar{w})\in \R\times \R^{m-1}\colon w_0\geqslant \enorm{\bar{w}}\}$ when $m>1$ and $\Lor^1\doteq \{w\in \R\colon w\geqslant 0\}$, or the Cartesian product of $r$ second-order cones in $\R^{m_i}$, with $i\in \{1,\dots,r\}$ and $m_1+\dots+m_r=m$;

\item  \hyperref[sec:sdp]{\textit{Nonlinear semidefinite programming}} (NSDP), which is obtained when $\E=\sym$ is the space of all $m\times m$ real symmetric matrices and $\K$ is the cone $\sym_+\doteq  \{W\in \sym\colon\forall d\in \R^m, d^{\T} Wd\geqslant 0\}$ of all positive semidefinite matrices, or a Cartesian product in the form $\K=\mathbb{S}^{m_i}_+\times \dots\times \mathbb{S}^{m_r}_+$, with $m_1+\dots+m_r=m$.
\end{itemize}

Both fields have grown independently and accumulated a large set of applications over the years, for example, in robust control~\cite{controllmi,controllmi2}, passive reduced-order modelling~\cite{freund2007}, structural optimization~\cite{robustness,structural}, the sphere covering problem~\cite{Mito2020}, and others (see~\cite{handbookconic,boydsocsurvey,handbookSDP} for a vast collection of examples). In conjunction, several algorithms have been developed for them, such as interior-point methods~\cite{benson2003,ipsdp,ipsocp}, sequential quadratic programming methods~\cite{sqpsdp,sqpsocp}, Newton-type methods~\cite{newtonsocp,KFF09}, and augmented Lagrangian methods~\cite{Santos2019,ahv,SDPNAL}, to name a few (see Yamashita and Yabe~\cite{yamashitasdpreview} for more details). Consequently, some theoretical aspects of NSOCP and NSDP, such as optimality conditions and regularity, have gained much relevance in the community as well. In particular, necessary optimality conditions are especially useful for giving theoretical global convergence support for iterative algorithms, in the sense that every feasible limit point of a given algorithm can be proven to satisfy some necessary optimality condition under a set of hypotheses. In fact, the reliability of an algorithm is deeply related with the strength of the optimality condition that supports its global convergence theory. From this point of view, second-order necessary optimality conditions improve the first-order ones by considering the curvature of the problem data over the set of directions where first-order information has little meaning, which is usually called \textit{cone of critical directions} (or \textit{critical cone}). Note that this kind of convergence theory is different from what is usually done for convex optimization problems, where \textit{second-order sufficient conditions} are used as convergence hypotheses. In the nonconvex case, the latter results in a local convergence analysis. Since the results of this paper are meant to be used in the aid of global convergence, we focus on necessary optimality conditions.

It is worth mentioning that second-order analysis in non-polyhedral conic contexts, such as NSOCP and NSDP, is considerably more intricate than in polyhedral contexts, such as in \textit{nonlinear programming} (NLP). This is justified by the fact that the curvature of $\K$ must be taken into account, besides the curvature of the functions defining the problem. The initial efforts to characterize this curvature were done by Kawasaki~\cite{Kawasaki1988}, whose results were generalized and refined by Cominetti~\cite{Cominetti1990}, and later completed by Bonnans, Cominetti, and Shapiro~\cite{bonn-comi-shap} with the notion of \textit{second-order regularity}. Then, Shapiro~\cite{Shapiro1997} obtained a specialized statement for it in the context of NSDP, that was later re-discovered by Forsgren~\cite{Forsgren2000}, Jarre~\cite{Jarre2012}, and Lourenço, Fukuda, and Fukushima~\cite{Lourenco2018}, using distinct nontrivial techniques that make each proof interesting on its own. For NSOCP, second-order necessary optimality conditions were first characterized by Bonnans and Ram{\'i}rez~\cite{bonnansramirez}, and later studied by Fukuda and Fukushima~\cite{Fukuda2016} who also presented sufficient conditions. In recent years, significant advances were obtained were obtained for very general classes of problems that have NSOCP and NSDP among their particular cases; see, for instance, the papers of Chieu et al.~\cite{Chieu2019} and Mohammadi, Mordukhovich, and Sarabi~\cite{Mohammadi2021}, both working under very weak assumptions, but while the former uses the classical notion of \textit{cone reducibility} introduced by~\cite{bonn-comi-shap}, the latter employs a new and more general concept called \textit{parabolic regularity} which allows obtaining second-order conditions by directly differentiating the indicator function in a particular sense instead of reducing the problem to remove its curvature at the point of interest. Thus, it is possible to say that the motivation for studying alternative ways of deriving second-order conditions for conic problems, in particular NSDP and NSOCP, has gone far beyond practical usage, but nevertheless we believe practice should not be ignored. 

With this in mind, some useful tools for proving new first- and second-order optimality conditions for optimization problems, which are deeply connected to the algorithmic approach, are the so-called \textit{sequential optimality conditions}. They were introduced in NLP, and later extended to NSOCP and NSDP, as KKT variants designed for building convergence theory of iterative algorithms (for details, we refer to the work of Andreani et al.~\cite{Santos2019,Andreani2020,ahm10,ahv,ams10}) and they gained some attention for being able to sharpen most convergence results for them in a general and unified manner (see, for instance,~\cite[Sec. 5.2]{CPG}). Also, a second-order sequential condition has recently appeared in the work of Andreani et al.~\cite{akkt2} for NLP, which not only provided an ideal way of incorporating second-order information in numerical methods, but also an intuitive strategy for building second-order analysis under weaker hypotheses than the traditional \textit{linear independence constraint qualification} (LICQ). These improvements were obtained by considering a somewhat ``weak'' second-order necessary optimality condition in the sense that only the lineality space of the critical cone is taken into account in their results. However, as it is well-known in NLP, this ``weak'' condition is the most suitable second-order condition for global convergence analysis of algorithms, since the stronger conditions that deal with the whole critical cone are not guaranteed to be fulfilled at the convergence points of a large class of algorithms, such as barrier-type methods~\cite{gould1999note} and augmented Lagrangian-type methods~\cite{andreani2018note}, even under very strong hypotheses. Besides, checking the validity of the ``strong'' second-order condition is an NP-hard class problem, whereas checking the ``weak'' condition is of polynomial class. Nevertheless, as far as we know, the latter condition has never received due attention in nonconvex conic contexts other than NLP.

Inspired by~\cite{akkt2}, we prove that every local minimizer satisfies the weaker version of the second-order necessary condition, for NSOCP and NSDP, but under weaker assumptions than all previous related works. In fact, the meaning of our results lies in the fact we assume neither \textit{nondegeneracy} nor \textit{strict complementarity}, since under these hypotheses the ``weak'' and ``strong'' second-order conditions are equivalent. Our approach is based on sequential conditions, which suggests that our results may be useful for proving convergence of algorithms to second-order stationary points of NSOCP and NSDP problems. We stress that even though NSOCP can be represented in terms of NSDP, it is interesting to discriminate them since the numerical methods designed to solve each problem might have different performances in practice \cite{surveysocp}. Also, it is not straightforward to derive second-order results for NSOCP only based on the NSDP results.

This paper is structured as follows: we begin by reviewing some classical results on first- and second-order optimality conditions for \eqref{NCP} and its particular cases with some degree of details, in Section~\ref{sec:prelim}. Then, we present our second-order analysis for NSOCP in Section~\ref{sec:socp}, and for NSDP in Section~\ref{sec:sdp}. At last, in Section \ref{sec:conclusion} we give some final considerations about this paper and related works.

\section{Technical background}\label{sec:prelim}

In this section, we introduce our notation and present some results from the literature that are directly related to ours. We also review in details some classical results on first- and second-order optimality conditions for NSOCP and NSDP.

We consider the standard inner product in $\R^n$, given by $\langle a,b \rangle\doteq \sum_{i=1}^n a_i b_i$, and the Euclidean norm, given by $\enorm{a}\doteq \sqrt{\langle a, a\rangle}$, for every $a,b\in \R^n$. The terms  $\interior(\K)$, $\bd(\K)$, and $\bd^+(\K)$ stand for the \textit{interior, boundary}, and \textit{boundary excluding the origin} of $\K$, respectively. Also, for any closed convex cone $C$, $\text{lin}(C)\doteq  C\cap (-C)$ denotes its \textit{lineality space}, which is the largest subspace contained in $C$. 

For a given finite indexed set $\{a_i\colon i\in \{1,\dots,k\}\}\subset \R$, we denote the array that has $a_i$ in its $i$-th position by $[a_i]_{i\in \{1,\dots,k\}}\in \R^k$ and, analogously, the matrix whose entries are the elements of $\{b_{ij}\colon i,j\in \{1,\dots,k\}\}\subset \R$ is denoted by $[b_{ij}]_{i,j\in \{1,\dots,k\}}\in \R^{k\times k}$. The identity matrix of $\R^{n\times n}$ is denoted by $\I_n$.  The \textit{gradient} and the \textit{Hessian} of a function $f \colon \Re^{n} \to \Re$ at an arbitrary point $x \in \Re^{n}$ are represented by $\grad f(x)$ and $\grad^2 f(x)$, respectively, and the first derivative of $g \colon \Re^{n} \to \E$ at $x$ is the linear mapping $Dg(x)[\cdot]\colon\R^n\to \E$ defined by the action $$Dg(x)[h]\doteq  \sum_{i=1}^n \partial_i g(x) h_i$$ for every $h\in \R^n$, where $\partial_i g(x)\in \E$ is the partial derivative of $g$ in the $i$-th variable, at $x$. In particular, if $\E=\R^m$ then $Dg(x)$ is exactly the Jacobian matrix of $g$ at $x$, in the canonical basis of $\R^m$; for instance, in this case the $i$-th row of $Dg(x)$ is given by the \textit{transpose} of $\nabla g_i(x)$, which is denoted by $\nabla g_i(x)^\T$, where $i\in\{1,\dots,m\}$. The \textit{adjoint} of $Dg(x)$ is the linear mapping $Dg(x)^*[\cdot]\colon \E\to \R^n$ such that $\langle Dg(x)[h],w\rangle=\langle h,Dg(x)^*[w]\rangle$ holds for every $h\in \R^n$ and every $w\in \E$, hence $$Dg(x)^*[w]=[\langle\partial_i g(x),w\rangle]_{i\in \{1,\dots,n\}},$$ for every $w\in \E$ and, if $\E=\R^m$ then $Dg(x)^*=Dg(x)^\T$. Similarly, we define the action of the linear mapping $D^2g(x)^*[\cdot]\colon \E\to\mathbb{R}^{n\times n}$ by $$D^2g(x)^*[w]\doteq  [\langle\partial_i\partial_j g(x),w \rangle]_{i,j\in \{1,\dots,n\}},$$ for every $w\in \E$.

The \textit{orthogonal projection} of $w\in \E$ onto $\K$ is the point $\Pi_{\K}(w)\in \K$ such that $$\norm{w-\Pi_{\K}(w)}=\min\{\norm{w-z}\colon z\in \K\}.$$ Note that $\Pi_{\K}(w)$ is well-defined as a convex function of $w$ since $\K$ is closed and convex. Also, a very useful fact is that every $w\in \E$ can be written as $$w=\Pi_{\K}(w)-\Pi_{\K}(-w),$$ with $\langle\Pi_{\K}(w),\Pi_{\K}(-w)\rangle=0$. This is commonly called the \textit{Moreau's decomposition} of $w$ and one of its many consequences is that $w\in \K$ if, and only if, $\Pi_{\K}(-w)=0$. Hence, the function $$\Phi(x)\doteq \frac{1}{2}\left(\enorm{h(x)}^2 + \norm{\Pi_{\K}(-g(x))}^2\right)$$ can be used as a measure of violation of the constraints of \eqref{NCP}, that is, a measure of infeasibility. A result by Fitzpatrick and Phelps~\cite[Thm. 2.2]{Fitzpatrick1982} can be employed to derive an expression for the gradient of $\Phi$ at $x$:

\begin{theorem}\label{thm:fitz}
For every $x\in \R^n$, we have
\jo{$}$\nabla \Phi(x)= Dh(x)^\T h(x)-Dg(x)^*[\pcon{-g(x)}].$\jo{$}
\end{theorem}

Also, we observe that $\nabla \Phi$ is a Lipschitz function, but it is not differentiable everywhere. In our analyses, we make use of its second derivative, which must be taken in the nonsmooth sense.

\subsection{Some elements of nonsmooth analysis}

Let $X$ and $Y$ be finite-dimensional normed linear spaces over $\R$. Let $F\colon X\to Y$ be a locally Lipschitz function and denote the set in which it is differentiable by $\mathcal{D}(F)$. The so-called \textit{B-subdifferential} of $F$ at a point $x\in X$, is the set of all limiting derivatives of $F$ at $x$, denoted by
$$\partial_B F(x)\doteq \left\{ V \in \mathcal{L}(X,Y) \colon \exists \seq{x}\subset \mathcal{D}(F), \ x^k\to x, \ D F(x^k)\to V \right\},$$ where $\mathcal{L}(X,Y)$ denotes the set of all linear mappings from $X$ to $Y$, and similarly to the previous section, $DF(x^k)\in \mathcal{L}(X,Y)$ denotes the first derivative of $F$ at $x^k$, for every $x^k\in \mathcal{D}(F)$. Evidently, for every $x$, the set $\partial_B F(x)$ is compact, and it is a singleton when $x\in \mathcal{D}(F)$, but it is not convex in general. Then, we also define the \textit{Clarke subdifferential} of $F$ at $x$, denoted by $\partial F(x)$, as the \textit{convex hull} of $\partial_B F(x)$, that is, $$\partial F(x)\doteq \text{Conv}(\partial_B F(x)).$$

In particular, when $X=\R^n$ and $Y=\R$, the \textit{generalized Hessian} of $F$ at $x$ is defined as 
$$\partial^2 F(x)\doteq \partial \nabla F(x),$$ which is the convex hull of the set of all limiting Hessian matrices of $F$ at $x$. Following Hiriart-Urruty et al.~\cite{Urruty1984}, the second-order necessary optimality condition for unconstrained minimizers of $F$ when it is differentiable, is the following:

\begin{theorem}\label{thm:nonsmooth2nd}
If $x\s$ is a local minimizer of a differentiable function $F\colon \R^n\to \R$ such that $\nabla F$ is locally Lipschitz, then $\nabla F(x\s)=0$, and for each $d\in \R^n$, there exists some $M\in \partial^2 F(x\s)$ such that $d^\T M d\geqslant 0$. In other words, $$\limsup_{d\in \R^n} d^\T \partial^2 F(x\s) d\geqslant 0.$$
\end{theorem}

We refer to~\cite[Thm. 3.1]{Urruty1984} for a proof. As observed by Hiriart-Urruty et al., it is not true that $d^\T M d\geqslant 0$ for all $M\in \partial^2 F(x\s)$, in general, and not even this holds for some fixed $M$ and all $d$. We employ this result to analyse the second derivative of $\Phi$, which is a \textit{nonsmooth-smooth} composition, so a chain rule is also required. There are several different extensions of the chain rule for subdifferentials, but the following result, by P{\'a}les and Zeidan~\cite{paleszeidan}, is enough for our purposes:

\begin{theorem}\label{thm:chainrule}
Fix some $x\in X$ and let $G\colon X\to Y$ and $F\colon Y\to Y$ be functions such that $G$ is continuously differentiable at $x$, and $F$ is Lipschitz in a neighborhood of $G(x)$. Then, we have $$\partial (F\circ G)(x)\subseteq \partial F(G(x))\circ DG(x),$$ where $\partial F(G(x))\circ DG(x)\doteq \{V\circ DG(x)\colon V\in \partial F(G(x))\}$.
\end{theorem}
\begin{proof}
The result follows from~\cite[Thm. 5.1]{paleszeidan} since it was originally proved for Banach spaces that satisfy the Radon-Nikodým property, which holds for every reflexive space, and every finite-dimensional space with a norm is reflexive.   
\end{proof}

Specificities about the subdifferential of the orthogonal projection onto the second-order and the semidefinite cone will be given in their respective sections.

\subsection{Necessary optimality conditions and constraint qualifications}

A \textit{constraint qualification} (CQ) is any assumption over the constraints at a feasible point $x$, that implies that the feasible set is similar to its first-order approximation around $x$. For instance, one of the most relevant ones is \textit{Robinson's CQ}~\cite{Robinson1976}, that holds at a feasible point $x$ when $Dh(x)$ has full row rank and there exists some $d\in \mathbb{R}^n$ such that\footnote{We use this characterization of Robinson's CQ as a definition because $\K$ is assumed to be self-dual, and consequently, to have nonempty interior.}
\begin{center}
$g(x)+Dg(x)[d]\in \interior(\mathcal{K})$\\
and\\
$Dh(x)d=0$.
\end{center}

It is widely known that Robinson's CQ is a generalization of the classical \textit{Mangasarian-Fromovitz constraint qualification} (MFCQ) from NLP. Such regularity condition allows one to study the optimality of a point in terms of the first-order approximation of the problem around it, that is, it is possible to prove that for every local solution $x\s$ of (\ref{NCP}) that satisfies Robinson's CQ, there exists some $\omega\s\in \mathcal{K}$ and some $\mu\s\in \R^p$ such that 

\begin{equation}\label{ncp:kktstat}
\nabla_x L(x\s, \omega\s,\mu\s)=0,
\end{equation}
and
\begin{equation}\label{ncp:kktcomp}
\inner{g(x\s)}{\omega\s}=0,
\end{equation}
where $$L(x, \omega,\mu)\doteq  f(x)-\langle g(x),\omega\rangle+\langle h(x),\mu\rangle$$ is the \textit{Lagrangian} function of \eqref{NCP} and $$\nabla_x L(x, \omega,\mu) \doteq \nabla f(x)-Dg(x)^*[\omega]+Dh(x)^\T\mu$$ is the gradient of $L(x,\omega,\mu)$ with respect to $x$. Equations \eqref{ncp:kktstat} and \eqref{ncp:kktcomp} compose the so-called \textit{Karush-Kuhn-Tucker} (KKT) conditions and, in this context, $\omega\s$ and $\mu\s$ are \textit{Lagrange multipliers} associated with $x\s$. Points that satisfy the KKT conditions are often called \textit{first-order stationary} or \textit{KKT points}. Condition \eqref{ncp:kktcomp} is often called \textit{complementarity}, and when additionally $-\omega\s$ belongs to the relative interior of $T_\K(g(x\s))\pol$, where
\si{\[T_{\K}(g(x\s))\doteq \{d\in \E\colon \exists \{d^k\}_{k\in \N}\to d, \ \exists \{\alpha^k\}_{k\in\N}\to 0, \ \forall k\in \N, \alpha^k>0, \ g(x\s)+\alpha^k d^k\in \K\}\]}
\jo{\[T_{\K}(g(x\s))\doteq \left\{
d\in \E\colon \begin{array}{l}
\exists \{d^k\}_{k\in \N}\to d, \ \exists \{\alpha^k\}_{k\in\N}\to 0, \\ \forall k\in \N, \alpha^k>0, \ g(x\s)+\alpha^k d^k\in \K
\end{array}
\right\}\]}
is the (\textit{Bouligand}) \textit{tangent cone} to $\K$ at $g(x\s)$, we say that \textit{strict complementarity} holds at the pair $(x\s,\omega\s)$~\cite[Def. 4.74]{bshapiro}. A relevant implication of Robinson's CQ is the boundedness of the set $\mathcal{M}(x\s)$ of all Lagrange multipliers associated with a local solution $x\s$. 

Second-order optimality conditions give extra information over the set of directions where first-order information is not meaningful. That is, we are interested in the set $$C(x\s)\doteq \{d\in \R^n\colon \langle\nabla f(x\s),d \rangle=0, \ Dh(x\s)^\T d=0, \ Dg(x\s)[d]\in T_{\K}(g(x\s))\},$$ which is the critical cone of (\ref{NCP}) at $x\s$. If Robinson's CQ holds at a local minimizer $x\s$ of \eqref{NCP}, then besides KKT, it also satisfies the \textit{basic second-order necessary condition} (BSOC), that is, for every $d\in C(x\s)$ there are Lagrange multipliers $\omega_d\s\in \K$ and $\mu_d\s\in \R^p$ such that \eqref{ncp:kktstat}, \eqref{ncp:kktcomp}, and
\begin{equation}\label{soncineq}
d^\T(\nabla_{x}^2L(x\s,\omega_d\s,\mu_d\s) + \sigma(x\s,\omega_d\s))d\geqslant 0
\end{equation}
hold, where $$\nabla_{x}^2L(x,\omega,\mu)\doteq  \nabla^2f(x)-D^2g(x)^*[\omega]+\sum_{i=1}^p \mu_i\nabla^2 h_i(x)$$ and $\sigma(x,\omega)$ is the so-called ``\textit{sigma-term}'', as presented by Cominetti~\cite[Thm. 4.1]{Cominetti1990}. In that paper, the author builds second-order conditions for (\ref{NCP}) based on the \textit{second-order tangent set} of $\K$ at $g(x)$ along $Dg(x)[d]$, that may be denoted by $T^2_{\K}(g(x),Dg(x)[d])$, and then establishes a ``dual form'' for it using the support function of $T^2_{\K}(g(x),Dg(x)[d])$, which is precisely the sigma-term. Hence, the sigma-term represents a possible curvature of $\K$ at $g(x)$, to some extent, and it can be proved that $\sigma(x,\omega)=0$ when $\K$ is \textit{polyhedral}, such as in NLP (for details, see~\cite{Cominetti1990}). In fact, the difficulty of second-order analysis in contexts more general than NLP lies almost entirely on the characterization of the sigma-term, which can be a very challenging task.

One of the major practical drawbacks of BSOC is that in order to verify whether it holds or not at a given point $x$, one must know the whole set $\mathcal{M}(x)$, which is not always possible. The stronger optimality condition where inequality~\eqref{soncineq} holds for every $d\in C(x\s)$, for some pair of multipliers $(\omega\s,\mu\s)$ (not depending on $d$), which is sometimes called the \textit{semi-strong necessary optimality condition}, does not present such a drawback. However, deciding the positivity of a matrix over a cone is an NP-hard class problem~\cite{nphard}, and so is checking the semi-strong condition.

A more practical alternative to BSOC and the semi-strong condition is the so-called \textit{weak second-order necessary condition} (WSOC), which is defined as follows:

\begin{definition}~\label{wsonc-def}
Let $x\s$ be a KKT point associated with some Lagrange multipliers $\omega\s\in \K$ and $\mu\s\in \R^p$. We say that WSOC holds at $x\s$ when 
\begin{equation}\label{wsoncineq}
d^\T(\nabla_{x}^2L(x\s,\omega\s,\mu\s) + \sigma(x\s,\omega\s))d\geqslant 0,
\end{equation}
for every $d\in S(x\s)\doteq\textnormal{lin}(C(x\s))$, which is the largest subspace contained in $C(x\s)$.
\end{definition}

Note that in Definition~\ref{wsonc-def} we only take directions in the subspace $S(x\s)$, called the \textit{critical subspace} of \eqref{NCP} at $x\s$, which coincides with $C(x\s)$ under strict complementarity\footnote{We will give a short proof for the fact $C(x\s)=\lin(C(x\s))$ under strict complementarity, for completeness: Let $d\in C(x\s)$ and suppose that there exists a Lagrange multiplier $-\omega\s$ in the relative interior of $T_{\K}(g(x\s))\pol$. Note that $\langle \nabla f(x\s),d\rangle=\langle\omega\s,Dg(x\s)[d]\rangle=0$ by the KKT conditions. Hence, $-\omega\s\in T_{\K}(g(x\s))\pol\cap \{Dg(x\s)[d]\}^\perp$, which implies $T_{\K}(g(x\s))\pol\subseteq \{Dg(x\s)[d]\}^\perp$ and, consequently, $\textnormal{span}(Dg(x\s)[d])\subseteq T_{\K}(g(x\s))$. Then, $Dg(x\s)[d]\in \lin(T_{\K}(g(x\s)))$, so $d\in \lin(C(x\s))$.}. At first sight, a second-order condition that only covers $S(x\s)$ instead of the whole $C(x\s)$ may seem disadvantageous in comparison with the semi-strong condition. In fact, the semi-strong condition implies WSOC. However, there are strong evidences that suggest that it is unlikely that BSOC or the semi-strong condition can be used to support the global convergence theory of any practical algorithm, unless $C(x\s)=S(x\s)$. In fact, for the particular case of NLP, Gould and Toint~\cite{gould1999note} presented a simple counterexample, with a quadratic objective function and a constraint of the form $x\geqslant 0$, for which a large class of barrier-type methods may produce an output sequence whose limit points fail to satisfy both BSOC and the semi-strong condition, even when every iterate of such sequence satisfies the second-order sufficient condition for its respective penalized problem. Later, Andreani and Secchin~\cite{andreani2018note} made a small modification in Gould and Toint's counterexample to obtain the same conclusion for augmented Lagrangian-type algorithms. WSOC, on the other hand, is guaranteed to be fulfilled under weak assumptions for some variants of the two methods we mentioned above~\cite{abms2,Prieto2003}, and also for a regularized SQP method for NLP~\cite{grSQP}. The negative conclusions regarding BSOC and the semi-strong condition have led some authors to doubt the existence of an algorithm that could be associated with a second-order condition that takes the whole critical cone into consideration. Following this discussion, Andreani et al.~\cite{akkt2} managed to characterize the weakest second-order constraint qualification that could guarantee the fulfilment of the semi-strong condition at the limit points of a large class of penalization-type algorithms that encompasses, for instance, all the aforementioned ones. However, such a constraint qualification was proven not to imply nor to be implied by LICQ~\cite[Ex. 4.5 and 4.6]{akkt2}, and to be violated even for box constraints. 

Despite the good algorithmic advantages of WSOC, Robinson's CQ alone is not enough to guarantee its fulfilment at local minimizers -- see, for instance, the counterexample by Baccari~\cite[Sec. 3]{Baccari2004} or the discussion in~\cite{Behling2017a}. Instead, the existing results on WSOC usually require a stronger CQ called \textit{nondegeneracy} (or \textit{transversality}), which holds at a feasible point $x$ when 

\begin{equation}\label{ncp:nondegen}
\E \times \R^p = \left(\text{lin}(T_{\K}(g(x)))+\text{Im}(Dg(x))\right)\times \text{Im}(Dh(x)).
\end{equation}
It was translated from differential equations to optimization by Shapiro and Fan~\cite{shapfan} and it is well-known that, for every nondegenerate solution $x\s$ of \eqref{NCP}, the set $\mathcal{M}(x\s)$ is a singleton, what resembles the effects of LICQ in NLP. Thus, nondegeneracy is analogous to LICQ, in this sense. 

\begin{theorem}\label{classicwsoc}
If $x\s$ is a local minimizer of \eqref{NCP} that satisfies nondegeneracy, then the KKT conditions hold at $x\s$ for some Lagrange multipliers $\omega\s\in \K$ and $\mu\s\in \R^p$ and, moreover, WSOC holds with respect to these multipliers. 
\end{theorem}

Note that Theorem \ref{classicwsoc} is simply a rephrasing of the necessity of BSOC after assuming uniqueness of the Lagrange multiplier (nondegeneracy), but we stated it as it is for comparison purposes since our main results consist of proving of Theorem \ref{classicwsoc} under less demanding conditions.

In the context of NLP, Andreani, Mart{\'i}nez, and Schuverdt~\cite{ams07} were able to prove Theorem~\ref{classicwsoc} replacing nondegeneracy (LICQ) with only MFCQ together with the so-called \textit{weak constant rank} (WCR) property, which holds at a feasible point $x\s$ when there exists a neighborhood $\mathcal{N}$ of $x\s$ such that 
\begin{equation}\label{nlp:wcr}
\{\nabla g_i(x)\}_{i\colon g_i(x\s)=0}\cup\{\nabla h_j(x)\}_{j\in \{1,\dots,p\}}
\end{equation}
has the same rank for every $x\in \mathcal{N}$. It is worth mentioning that WCR is not a CQ on its own~\cite[Ex. 5.1]{ams07} and that the joint condition ``MFCQ+WCR'' was proven to be strictly weaker than LICQ~\cite[Ex. 5.2]{ams07}. Later, a simpler proof of this result was presented by Andreani et al.~\cite[Crlr. 4.3 and Thm. 4.1]{akkt2}, using sequential optimality conditions. In the following sections, we generalize the WCR property and the result of~\cite{ams07} for NSOCP and NSDP, using an approach similar to~\cite{akkt2}.

As a matter of fact, Andreani, Echagüe, and Schuverdt \cite{aes2} presented a result similar to Theorem \ref{classicwsoc}, but under Janin's \textit{constant rank constraint qualification} (CRCQ)~\cite{janin}, which is also weaker than LICQ and independent of ``MFCQ+WCR''. However, extending constant rank-type CQs to conic contexts is not easy, and finding an extension that preserves all of its interesting properties is even more difficult. In fact, there is a series of papers by Andreani et al.~\cite{naive-crcq,seqcrcq,facialcrcq,seqcrcq-socp} presenting distinct extensions of CRCQ for NSDP and NSOCP that suit distinct applications. For instance,~\cite{seqcrcq,seqcrcq-socp} deal with convergence of algorithms to first-order stationary points but no second-order properties were proven, whereas~\cite{facialcrcq} presents a more geometric approach with some interesting theoretical properties but no application towards algorithms was provided. We should mention, nevertheless, that the extension of WCR presented this paper is not a particular case of any of the conditions from the aforementioned papers.


\section{Second-order cone programming}
\label{sec:socp}

The standard NSOCP problem can be seen as a particular case of \eqref{NCP} where $\E=\R^m$ and $\K\doteq  \K_{1}\times \dots \times \K_{r}$ is a Cartesian product of Lorentz cones, that is, $\K_i\doteq \Lor^{m_i}$ for all $i\in \{1,\dots,r\}$, where $m_1 + \dots + m_r=m$ and $\K_{i}\subset\R^{m_i}$. In this section, we consider $\R^m$ with its standard inner product and the Euclidean norm. The notation $w = (w_0, \bar{w})$ refers to a partition of $w\in \R^{m_i}$ where $w_0 \in \Re$ is its first entry and $\bar{w} \in \Re^{m_i-1}$ is the subvector with the remaining entries. To make the NSOCP problem explicit, define $g \doteq (g_1,\dots,g_r)$ with $g_i
\colon \Re^n \to \Re^{m_i}$ for every $i\in\{1,\dots,r\}$, and obtain

\begin{equation}
  \tag{NSOCP}
  \begin{aligned}
    & \underset{x \in \mathbb{R}^{n}}{\text{Minimize}}
    & & f(x), \\
    & \text{subject to}
    & & g_i(x) \in  \K_{i}, \forall i\in\{1,\dots,r\} \\
    & & & h(x) = 0.
  \end{aligned}
  \label{NSOCP}
\end{equation}


As usual in the study of \eqref{NSOCP}, given a feasible point $x$, we define the following sets of indices, which constitute a
partition of $\{1,\dots,r\}$:
\begin{equation}
  \label{socp:index}
  \begin{array}{r@{\:}c@{\:}l}
    I_0(x) & \doteq   & \{ i \in \{1,\dots,r\} \colon 
    g_i(x) = 0 \}, \\
    I_B(x) & \doteq   & \{ i \in \{1,\dots,r\} \colon 
    g_i(x) \in \bd^+(\K_i) \},\\
    I_I(x) & \doteq   & \{ i \in \{1,\dots,r\} \colon 
    g_i(x) \in \interior(\K_i) \}.
  \end{array}
\end{equation}

Moreover, when we are dealing with a KKT point $x\s$ associated with Lagrange multipliers $\omega\s\in\K$ and $\mu\s\in \R^p$, we consider the subset of $I_B(x\s)$ given by $$I_{BB}(x\s,\omega\s) \doteq  \{ i  \in I_B(x\s) \colon \omega_i\s \in \bd^+(\K_i) \}$$  and the critical subspace of \eqref{NSOCP} at $x\s$ can be written in terms of such indices, as follows:

\begin{equation} \label{socp:critsubspace}
\si{S(x\s)= \left\{ d \in \Re^n \colon \: Dg_i(x\s)d = 0, i \in I_0(x\s); \:\: g_i(x\s)^\T \Gamma_i Dg_i(x\s)d = 0, i \in I_B(x\s); \:\: Dh(x\s)d = 0 \right\}}
\jo{S(x\s)=\left\{ d \in \Re^n \colon \:
\begin{array}{l}
Dh(x\s)d=0; \:\: Dg_i(x\s)d = 0, i \in I_0(x\s); \\ 
g_i(x\s)^\T \Gamma_i Dg_i(x\s)d = 0, i \in I_B(x\s) 
\end{array}
\right\}},
\end{equation}
where
\begin{equation}
  \label{socp:gamma}
  \Gamma_i \doteq   
  \left[
    \begin{array}{cc}
      1 & 0^\T \\
      0 & -\I_{m_i-1}
    \end{array}
  \right] \in \Re^{m_i \times m_i}
  \quad \mbox{for all } i\in\{1,\dots,r\}.
\end{equation} 

The sigma-term at $x\s$, when specialized to \eqref{NSOCP}, can be written as
\begin{equation*}
  \label{socp:sigmafull}
\sigma(x\s,\omega\s)=\sum_{i\in I_{BB}(x\s,\omega\s)} \sigma_i(x\s,\omega\s),
\end{equation*}
where
\begin{equation}
\label{socp:sigma}
\sigma_i(x\s,\omega\s) =   
  -\frac{[\omega\s_i]_0}{[g_{i}(x\s)]_0} Dg_i(x\s)^\T \Gamma_i Dg_i(x\s), \quad \text{for all } i\in I_{BB}(x\s,\omega\s).
\end{equation}
We refer to~\cite{Fukuda2016} for details. 

Also, the specialized characterization of the nondegeneracy condition in NSOCP, following Bonnans and Ram{\'i}rez~\cite[Prop. 19]{bonnansramirez}, can be written as follows:

\begin{proposition}\label{socp:nondegen}
  Let $x\s \in \Re^n$ be a feasible point of~\eqref{NSOCP}. The nondegeneracy condition holds at
  $x\s$ if, and only if, the set
  \begin{equation}\label{socp:ndgset}
  \big\{\nabla h_{i}(x\s)\big\}_{i\in \{1,\dots,p\}}\bigcup\big\{\nabla g_{ij}(x\s)\big\}_{\substack{i\in I_0(x\s)\\ j\in \{1,\dots,m_i\}}} \bigcup
  \big\{Dg_i(x\s)^\T\Gamma_i\tilde g_i(x\s)\big\}_{i\in I_B(x\s)}
  \end{equation}
  is linearly independent, where $\nabla g_{ij}(x)$ denotes the transpose of the $j$-th row of $Dg_i(x)$ and
  \begin{equation}
    \label{socp:eq-tilde-g}
    \tilde g_i(x)\doteq  (\enorm{\overline{g_i(x)}},\overline{g_i(x)}).
  \end{equation}
\end{proposition}

In~\cite[Def. 3.3]{Santos2019}, the authors extend a sequential optimality condition called \textit{Approximate-KKT} (AKKT) from NLP~\cite{ahm10} to the NSOCP context. In short, AKKT is a punctual necessary optimality condition that also incorporates a bit of local information. That is, every point $x\s$ that satisfies AKKT (though not necessarily KKT) is accompanied by a sequence $\seq{x}\to x\s$ such that each $x^k$ approximately satisfies the KKT conditions with some approximate Lagrange multipliers $\omega^k$ and $\mu^k$. Since our analyses are based on AKKT, we now recall its definition and some of its properties.

\begin{definition}[AKKT for NSOCP]\label{def:socpakkt}
A feasible point $x\s$ of \eqref{NSOCP} satisfies the AKKT condition when there exist sequences $\seq{x}\to x\s$, $\seq{\omega}\subset \K$, and $\seq{\mu}\subset \R^p$ such that 
\begin{equation}\label{socp:akktstat}
\nabla_x L(x^k,\omega^k,\mu^k)\to 0
\end{equation}
and
\begin{equation}\label{socp:akktcomp}
\begin{aligned}
i\in I_I(x\s)\Rightarrow\omega_i^k\to 0,\\
i\in I_B(x\s)\Rightarrow \omega_i^k\to 0 \text{ or } \omega_i^k\in \bd^+(\K_i) \text{ with } \frac{\bar{\omega}^k_i}{\enorm{\bar{\omega}^k_i}}\to \frac{\bar{g_i(x\s)}}{\enorm{\bar{g_i(x\s)}}}.
\end{aligned}
\end{equation}
\end{definition}

It was proved in~\cite[Thm. 3.1]{Santos2019} that AKKT is indeed a genuine necessary optimality condition independently of CQs, in contrast with KKT. Also, their proof is constructive, which means it tells us how to obtain the sequences of perturbed KKT points and multipliers. Next, we state their result with a slightly different phrasing, in order to highlight such construction.

\begin{theorem}\label{socp:minakkt}
Let $x\s$ be a local minimizer of \eqref{NSOCP}. Then, for any given sequence $\{\rho_k\}_{k\in \N}\to +\infty$, there exists a sequence $\seq{x}\to x\s$, such that each $x^k$ is a local minimizer of the regularized penalty function 
$$F_k(x) \doteq   f(x)+\frac{1}{4}\enorm{x-x\s}^4 
  + \frac{\rho_k}{2} \left( \sum_{i=1}^r \enorm{\Pi_{\K_i}(-g_i(x))}^2
  + \enorm{h(x)}^2 \right).$$
Also, the multiplier sequences given by $\omega^k_i \doteq   \rho_k \Pi_{\K_i}(-g_i(x^k))$ for all $i\in\{1,\dots,r\}$ and $\mu^k \doteq   \rho_k h(x^k)$ satisfy \eqref{socp:akktstat} and \eqref{socp:akktcomp} together with $\seq{x}$. Consequently, $x\s$ satisfies AKKT.
\end{theorem}

A key property of AKKT, as stated in~\cite[Thm. 3.1]{Santos2019}, is that the sequences of multipliers from Definition~\ref{def:socpakkt} must be bounded when $x\s$ satisfies Robinson's CQ. Hence, AKKT implies KKT under Robinson's CQ. Also, in the same paper the authors present a variant of the classical \textit{Powell-Hestenes-Rockafellar} (PHR) Augmented Lagrangian method (see~\cite{Hestenes,Powell,Rockafellar}) and prove that its output sequences can be fully described by AKKT. 


\subsection{Second-order optimality conditions}
\label{sec:akkt2}

Here, we build second-order analysis for \eqref{NCP} primarily under Robinson's CQ instead of nondegeneracy and strict complementarity, but since Robinson's CQ alone is not enough to complete that task~\cite{Baccari2004}, we also introduce a generalized version of the WCR property.

\begin{definition}[WCR for NSOCP]\label{socp:wcrdef}
  \si{Let $x\s \in \Re^n$ be a feasible point of~\eqref{NSOCP}. }We say
  that the \emph{weak constant rank} property is satisfied at \jo{ a feasible point }$x\s$\jo{ of \eqref{NSOCP} } if there exists a neighborhood $\mathcal{N}$ of $x\s$ such that the set
  \begin{equation}\label{socp:wcrset}
  \big\{\nabla h_{i}(x)\big\}_{i\in \{1,\dots,p\}}\bigcup\big\{\nabla g_{ij}(x)\big\}_{\substack{i\in I_0(x\s)\\ j\in \{1,\dots,m_i\}}} \bigcup
  \big\{Dg_i(x)^\T\Gamma_i\tilde g_i(x)\big\}_{i\in I_B(x\s)}
  \end{equation}
  has the same rank, for all $x \in \mathcal{N}$.
\end{definition}

In view of the characterization of nondegeneracy for NSOCP provided by Proposition \ref{socp:nondegen}, we see that nondegeneracy implies both Robinson's CQ and WCR in this context, just as in the NLP case. On the other hand, \cite[Ex. 5.2]{ams07} exhibits a point that satisfies MFCQ and WCR, but not LICQ. Hence, the joint condition ``Robinson's CQ+WCR'' is strictly weaker than nondegeneracy. 

The main feature of the WCR property in NLP is its effect on the continuity of perturbations of the critical subspace around a feasible point $x\s$. Next, we prove that this property is maintained in \eqref{NSOCP}.  

\begin{lemma}
  \label{socp:wcr-inner}
  Let $x\s \in \Re^n$ be a feasible point of~\eqref{NSOCP}. Then, the
  WCR property holds at $x\s$ if, and only if, the set-valued mapping $x\mapsto
  S(x,x\s)$ is inner semicontinuous at $x\s$, where
  \begin{equation}
    \label{socp:pert-subspace}
    \si{S(x,x\s) \doteq   \left\{ d \in \Re^n \colon \: Dh(x)d = 0; \:\: \forall i \in I_0(x\s),Dg_i(x)d = 0; \:\:
      \forall i \in I_B(x\s),\tilde{g}_i(x)^\T \Gamma_i Dg_i(x)d = 0\right\},}
      \jo{S(x,x\s) \doteq \left\{ d \in \Re^n \colon \: 
      \begin{array}{l}
      Dh(x)d = 0; \:\: \forall i \in I_0(x\s),Dg_i(x)d = 0; \\
      \forall i \in I_B(x\s),\tilde{g}_i(x)^\T \Gamma_i Dg_i(x)d = 0
	  \end{array}      
      \right\},}
  \end{equation}
  and $\tilde{g}_i$ is defined in~\eqref{socp:eq-tilde-g}.
\end{lemma}

\begin{proof}
Following the steps of the proof of~\cite[Prop. 2]{WCR}, we see that \cite[Thm. 1.1.8]{aubinf} tells us that $x\mapsto S(x,x\s)$ is inner semicontinuous at $x\s$ if, and only if, the set-valued mapping $x\mapsto S(x,x\s)\pol$ is outer semicontinuous at $x\s$, where $S(x,x\s)\pol\doteq -S(x,x\s)^*$ denotes the \textit{polar} of $S(x,x\s)$. Since in this case we have 
\begin{equation*}\si{S(x,x\s)\pol=\left\{\sum_{i\in I_0(x\s)} Dg_{i}(x)^\T a_{i} + \sum_{i=1}^p \nabla h_i(x) b_i + \sum_{i\in I_B(x\s)} Dg_i(x)^\T\Gamma_i\tilde g_i(x) c_i \ \colon \ a_{i}\in \R^{m_i} , \ b_i,c_i\in \R\right\},}
\jo{S(x,x\s)\pol=\left\{
\begin{aligned}
\sum_{i\in I_0(x\s)} Dg_{i}(x)^\T a_{i} + \sum_{i=1}^p \nabla h_i(x) b_i +\\
+\sum_{i\in I_B(x\s)} Dg_i(x)^\T\Gamma_i\tilde g_i(x) c_i\\
\end{aligned}
\ \colon \ a_{i}\in \R^{m_i} , \ b_i,c_i\in \R\right\},}
\end{equation*}
the result follows directly from~\cite[Prop. 3.2.9]{fachpang}.   
\end{proof}

As in NLP, the subspace $S(x,x\s)$ may be called \textit{perturbed critical subspace} of \eqref{NSOCP} at $x$, around $x\s$. 
The last ingredient we need for the main theorem of this section is an explicit characterization of the subdifferential of the projection onto $\K_i$. In order to present that, 
for each $i\in\{1,\dots,r\}$, let $M_i \colon \Re \times \Re^{m_i-1} \to
  \Re^{m_i \times m_i}$ be defined as
  \begin{displaymath}
    M_i(\xi,w) \doteq  
    \frac{1}{2}
    \left[
      \begin{array}{cc}
        1 &  w^\T \\
        w & (1+\xi) \I_{m_i-1} - \xi w w^\T \\
      \end{array}
    \right]
  \end{displaymath}
and observe that the matrix $M_i(\xi,u)$ is symmetric positive
semidefinite whenever $\vert\xi\vert \le 1$ and $\enorm{w} = 1$ \cite[Lem. 2.8]{KFF09}. 

The following lemma, that can be found in~\cite[Lem. 14]{PSS03} and~\cite[Prop. 4.8]{HYF05}, provides a description of the B-subdifferential of the projection onto $\K_i$, in terms of $M_i(\xi,w)$.

\begin{lemma}
  \label{socp:lemma-subdiffproj}
  
  The B-subdifferential $\partial_B \Pi_{\K_i} (z)$ of the orthogonal projection onto $\K_i$ at $z\in \R^{m_i}$ is given as follows:
  \begin{itemize}
  \item[(a)] If $z \in \interior(-\K_{i})$, then $\partial_B
    \Pi_{\K_i}(z) = \big\{ 0 \big\}$;
  \item[(b)] If $z \in \interior(\K_{i})$, then $\partial_B
    \Pi_{\K_i}(z) = \big\{ \I_{m_i} \big\}$;
  \item[(c)] If $z \notin \K_i \cup (-\K_{i})$, then
    $\displaystyle{\partial_B \Pi_{\K_i}(z) = \left\{ M_i \left(
          \frac{z_0}{\enorm{\bar{z}}}, \frac{\bar{z}}{\enorm{\bar{z}}}
        \right) \right\} }$;
  \item[(d)] If $z \in \bd^+(\K_{i})$, then
    $\displaystyle{\partial_B \Pi_{\K_i}(z) = \left\{ \I_{m_i}, M_i
        \left( 1, \frac{\bar{z}}{\enorm{\bar{z}}} \right) \right\}}$;
  \item[(e)] If $z \in \bd^+(-\K_{i})$, then
    $\displaystyle{\partial_B \Pi_{\K_i}(z) = \left\{ 0, M_i \left( -1,
          \frac{\bar{z}}{\enorm{\bar{z}}} \right) \right\}}$;
  \item[(f)] If $z = 0$, then $\displaystyle{\partial_B \Pi_{\K_i}(z) =
      \{ 0, \I_{m_i} \} \cup \{ M_i(\xi,w) \colon \vert\xi\vert \le 1, \enorm{w}
      = 1 \}}$.
  \end{itemize}
\end{lemma}

To the best of our knowledge, the first specialized study on second-order necessary conditions for \eqref{NSOCP} is credited to Bonnans and Ram{\'i}rez~\cite[Thm. 30]{bonnansramirez}, where they assume nondegeneracy and the so-called \textit{second-order growth condition} (or \textit{uniform growth condition}). Fukuda and Fukushima~\cite[Thm. 4.5]{Fukuda2016} also developed second-order conditions via squared slack variables, under nondegeneracy and strict complementarity. Our contribution to this discussion is to draw attention to the fact that the nondegeneracy assumption can be strictly weakened and that strict complementarity is not necessary when considering WSOC, which is also the main result of this section.

\begin{theorem}\label{socp:minwsoc}
Let $x\s$ be a local minimizer of \eqref{NSOCP} satisfying Robinson's CQ and the WCR property. Then, there are some Lagrange multipliers $\omega\s\in \K$ and $\mu\s\in \R^p$ such that the KKT conditions and WSOC hold.
\end{theorem}

\begin{proof}
Let $x\s$ be a local minimizer of \eqref{NSOCP}. Then, by Theorem~\ref{socp:minakkt}, for any given $\{\rho_k\}_{k\in \N}\to +\infty$, there exists a sequence $\seq{x}\to x\s$ such that $x^k$ is a local minimizer of $F_k(x)$ for each~$k$, where
  \[
  F_k(x) =   f(x)+\frac{1}{4}\enorm{x-x\s}^4 
  + \frac{\rho_k}{2} \left( \sum_{i=1}^r \enorm{\Pi_{\K_i}(-g_i(x))}^2
  + \enorm{h(x)}^2 \right).
  \]
  From the local optimality of $x^k$, we obtain
  \[
  \si{\nabla F_k(x^k) = \nabla f(x^k)+\enorm{x^k-x\s}^2(x^k-x\s)
  - \sum_{i=1}^rDg_i(x^k)^\T\rho_k\Pi_{\K_i}(-g_i(x^k))
   +\rho_k Dh(x^k)^\T h(x^k)=0}
  \jo{\begin{aligned}
  \nabla F_k(x^k) = \nabla f(x^k)+\enorm{x^k-x\s}^2(x^k-x\s)
  - \hspace{4cm}\\ -\sum_{i=1}^rDg_i(x^k)^\T\rho_k\Pi_{\K_i}(-g_i(x^k))
   +\rho_k Dh(x^k)^\T h(x^k)=0
  \end{aligned}}
  \]
  and by Theorem~\ref{thm:nonsmooth2nd}, for every $d\in \R^n$ and every $i\in\{1,\dots,r\}$, there exists some $\chi_i^k\in \partial(\Pi_{\K_i}\circ -g_i)(x^k)$ such that $d^\T \nabla^2F_k(x^k) d\geqslant 0,$ where we denote by $\nabla^2F_k(x^k)$ the element of the generalized Hessian of $F_k$ at $x^k$ that is defined in terms of $\chi_i^k$, by an abuse of notation. That is,
  
  \begin{eqnarray*}
  \nabla^2F_k(x^k) & \doteq & \nabla^2f(x^k)+\enorm{x^k-x\s}^2\,\I_n + 2(x^k-x\s)(x^k-x\s)^\T \\
  & & {} - \sum_{i=1}^r \Bigg( \sum_{j=1}^{m_i}(\rho_k\Pi_{\K_i}(-g_i(x^k)))_j\nabla^2g_{ij}(x^k)
    -\rho_kDg_i(x^k)^\T \chi_i^k \Bigg) \\
  & & {} + \sum_{j=1}^p \left( \rho_k h_j(x^k) \grad^2 h_j(x^k) 
    + \rho_k \grad h_j(x^k) \grad h_j(x^k)^\T \right).
  \end{eqnarray*}
  
Following Theorem~\ref{socp:minakkt}, we define $\omega^k_i \doteq   \rho_k \Pi_{\K_i}(-g_i(x^k))$ for all $i\in\{1,\dots,r\}$, and $\mu^k \doteq   \rho_k h(x^k)$ for every $k\in \N$, which satisfy \eqref{socp:akktstat}. Also, it follows from Theorem~\ref{thm:chainrule}  that there exists some $V^{k}_i\in\partial\Pi_{\K_i}(-g_i(x^k))$ such that \jo{$}$\chi_i^k = V^k_i \circ -Dg_i(x^k)=-V^k_i Dg_i(x^k),$\jo{$} where $\circ$ denotes a composition of linear operators. Hence, the expression $d^\T \nabla^2 F_k(x^k)d\geqslant 0$ can be rewritten as
  \begin{equation}\label{socp:akkt2ex}
  \si{d^\T\left(\nabla_x^2L(x^k,\omega^k,\mu^k)+\rho_k\sum_{i=1}^r Dg_i(x^k)^\T
  V^{k}_iDg_i(x^k) +\rho_k\sum_{j=1}^p \grad
      h_j(x^k) \grad h_j(x^k)^\T\right)d \geqslant -d^\T\Delta^k d,}
  \jo{\begin{aligned}
  d^\T\Bigg( \nabla_x^2L(x^k,\omega^k,\mu^k)+\rho_k\sum_{i=1}^r Dg_i(x^k)^\T
  V^{k}_iDg_i(x^k) +\Bigg. \hspace{3cm} \\ \Bigg.+\rho_k\sum_{j=1}^p \grad
      h_j(x^k) \grad h_j(x^k)^\T \Bigg)d \geqslant -d^\T\Delta^k d,
      \end{aligned}}
  \end{equation}
  where $\Delta^k\doteq  \enorm{x^k-x\s}^2\I_n + 2(x^k-x\s)(x^k-x\s)^\T \to0$.
  
Under Robinson's CQ, the sequence $\{(\omega^k,\mu^k)\}_{k\in \N}$ is bounded (see the proof of~\cite[Thm. 3.3]{Santos2019}). Then, for every limit point $(\omega\s,\mu\s)$ of $\{(\omega^k,\mu^k)\}_{k\in \N}$, note that $x\s$ satisfies the KKT conditions. Without loss of generality, we assume $\{(\omega^k,\mu^k)\}_{k\in \N} \to (\omega\s,\mu\s)$. Now, from WCR and Lemma~\ref{socp:wcr-inner}, we know that the mapping $x\mapsto S(x,x\s)$ as in~\eqref{socp:pert-subspace} is inner semicontinuous at $x\s$, then for each $d\in S(x\s)$ there exists a sequence $\seq{d}\to d$ such that $d^k\in S(x^k,x\s)$ for all $k\in \N$. 
  
  For each $i\in\{1,\dots,r\}$, define
  \[
  u_i^k=\Big([u_i^k]_0,\overline{u_i^k}\Big) \doteq   Dg_i(x^k)d^k.
  \]
Our next step is to compute $\rho_k(u^k_i)^\T V^{k}_i u^k_i$ and its limit points in three independent cases:
\begin{enumerate}  
 \item If $i\in I_I(x\s)$, we have $g_i(x^k)\in \interior(\K_i)$ for all $k$
  sufficiently large. Then, from
  Lemma~\ref{socp:lemma-subdiffproj} item (a), $V_i^{k}=0$ and $\rho_k(u^k_i)^\T V^{k}_i u^k_i= 0$ for such $k$;
  
  \item If $i\in I_0(x\s)$, recalling that $d^k \in S(x^k,x\s)$, we have
  $u_i^k=0$ for all $k\in \N$, which means $\rho_k(u^k_i)^\T V^{k}_i u^k_i=0$ in this case as well;
  
  \item If $i\in I_B(x\s)$, the sequence $\{g_i(x^k)\}_{k\in \N}$ can be essentially split into three subsequences, which have distinct influences over $\rho_k(u^k_i)^\T V^{k}_i u^k_i$. Hence, they are separately analysed below, where $N_1, N_2$, and $N_3$ constitute a partition of $\N$:
  
\begin{enumerate}
\item [(i)] $\{g_i(x^k)\}_{k\in N_1}\subset \interior(\K_i)$. Here, $\omega^k_i=0$ for every $k\in N_1$. Also, by item (a) of Lemma~\ref{socp:lemma-subdiffproj}, $V^k_i=0$ and $\rho_k(u^k_i)^\T V^{k}_i u^k_i=0$ for every $k\in N_1$;

\item [(ii)] $\{g_i(x^k)\}_{k\in N_2}\subset \R^m\setminus (\K_i\cup -\K_i)$. From Lemma~\ref{socp:lemma-subdiffproj} item (c) we obtain
  \[V_i^{k}=M_i\left(-\frac{[g_i(x^k)]_0}{\enorm{\overline{g_i(x^k)}}},
    -\frac{\overline{g_i(x^k)}}{\enorm{\overline{g_i(x^k)}}}\right)
  \] which can be explicitly written as $$V_i^{k}=\displaystyle{\frac{1}{2}}\left[
      \begin{array}{cc}
        1 & \displaystyle{-\frac{\overline{g_i(x^k)}^\T}{\enorm{\overline{g_i(x^k)}}}} \\[10pt] 
        \displaystyle{-\frac{\overline{g_i(x^k)}}{\enorm{\overline{g_i(x^k)}}}} & 
        \si{\displaystyle{\bigg(1-\frac{[g_i(x^k)]_0}{\enorm{\overline{g_i(x^k)}}}\bigg)\I_{m_i-1}
          +\bigg(\frac{[g_i(x^k)]_0}{\enorm{\overline{g_i(x^k)}}}\bigg)
          \frac{\overline{g_i(x^k)}\:\:\overline{g_i(x^k)}^\T}{\enorm{\overline{g_i(x^k)}}^2}}}
        \jo{Z^k}
      \end{array}\right],$$
      \jo{where $$Z^k\doteq \displaystyle{\bigg(1-\frac{[g_i(x^k)]_0}{\enorm{\overline{g_i(x^k)}}}\bigg)\I_{m_i-1}
          +\bigg(\frac{[g_i(x^k)]_0}{\enorm{\overline{g_i(x^k)}}}\bigg)
          \frac{\overline{g_i(x^k)}\:\:\overline{g_i(x^k)}^\T}{\enorm{\overline{g_i(x^k)}}^2}},$$}
     and it is elementary to see that
    \begin{equation*}
    \si{(u_i^k)^\T V^{k}_iu_i^k = \displaystyle{\frac{1}{2}\Bigg([u_i^k]_0^2-
      \frac{2[u_i^k]_0\overline{g_i(x^k)}^\T\bar{u_i^k}}{\enorm{\overline{g_i(x^k)}}}
      + \bigg(1-\frac{[g_i(x^k)]_0}{\enorm{\overline{g_i(x^k)}}}\bigg)\enorm{\bar{u_i^k}}^2
      +\frac{[g_i(x^k)]_0 (\overline{g_i(x^k)}^\T\bar{u_i^k})^2}{\enorm{\overline{g_i(x^k)}}^3}\Bigg)}.}
      \jo{\begin{split}(u_i^k)^\T V^{k}_iu_i^k = \frac{1}{2}\Bigg([u_i^k]_0^2-
      \frac{2[u_i^k]_0\overline{g_i(x^k)}^\T\bar{u_i^k}}{\enorm{\overline{g_i(x^k)}}}
      + \bigg(1-\frac{[g_i(x^k)]_0}{\enorm{\overline{g_i(x^k)}}}\bigg)\enorm{\bar{u_i^k}}^2
      +\Bigg. \\ \Bigg. 
      +\frac{[g_i(x^k)]_0 (\overline{g_i(x^k)}^\T\bar{u_i^k})^2}{\enorm{\overline{g_i(x^k)}}^3}\Bigg).\end{split}}
      \end{equation*}
      Also, since $d^k\in S(x^k,x\s)$ and $i\in I_B(x\s)$ we have
  $\tilde{g}_i(x^k)^\T\Gamma_i u_i^k=0$, or equivalently,
  $\overline{g_i(x^k)}^\T\bar{u_i^k} = \enorm{\overline{g_i(x^k)}}[u_i^k]_0$. Replacing this in the above expression, we obtain:
  \si{\begin{align}
    (u_i^k)^\T V^{k}_iu_i^k 
    & = \displaystyle{\frac{1}{2}\left([u_i^k]_0^2-2[u_i^k]_0^2
      + \bigg(1-\frac{[g_i(x^k)]_0}{\enorm{\overline{g_i(x^k)}}}\bigg)\enorm{\bar{u_i^k}}^2
      + \frac{[u_i^k]_0^2[g_i(x^k)]_0}{\enorm{\overline{g_i(x^k)}}}\right)} \nonumber \\[5pt]
    & = \displaystyle{-\frac{1}{2} \left(1-\frac{[g_i(x^k)]_0}{\enorm{\overline{g_i(x^k)}}}\right)
      \Big([u_i^k]_0^2-\enorm{\bar{u_i^k}}^2\Big)}.\label{socp:uVu}
  \end{align}}
  \jo{\begin{align}
    & \ (u_i^k)^\T V^{k}_iu_i^k \nonumber\\  
    = & \  \displaystyle{\frac{1}{2}\left([u_i^k]_0^2-2[u_i^k]_0^2
      + \bigg(1-\frac{[g_i(x^k)]_0}{\enorm{\overline{g_i(x^k)}}}\bigg)\enorm{\bar{u_i^k}}^2
      + \frac{[u_i^k]_0^2[g_i(x^k)]_0}{\enorm{\overline{g_i(x^k)}}}\right)} \nonumber \\[5pt]
    = & \ \displaystyle{-\frac{1}{2} \left(1-\frac{[g_i(x^k)]_0}{\enorm{\overline{g_i(x^k)}}}\right)
      \Big([u_i^k]_0^2-\enorm{\bar{u_i^k}}^2\Big)}.\label{socp:uVu}
  \end{align}
  }

  It follows from our specific choice of approximate multiplier that
  \begin{equation}\nonumber
  \omega^k_i=\rho_k\Pi_{\K_i}(-g_i(x^k))
  =\rho_k\frac{\enorm{\overline{g_i(x^k)}}-[g_i(x^k)]_0}{2}
  \left(1,-\frac{\overline{g_i(x^k)}}{\enorm{\overline{g_i(x^k)}}}\right).
  \end{equation}
  Hence, we have
  \[
  \frac{[\omega_i^k]_0}{\enorm{\overline{g_i(x^k)}}}=\frac{\rho_k}{2} 
  \left(1-\frac{[g_i(x^k)]_0}{\enorm{\overline{g_i(x^k)}}}\right)
  \] 
  and from~\eqref{socp:uVu}, we obtain 
  \begin{equation}\label{socp:sigmaborder}
  \rho_k (u_i^k)^\T V^{k}_iu_i^k = -\frac{[\omega_i^k]_0}{\enorm{\overline{g_i(x^k)}}}
  \Big([u_i^k]_0^2-\enorm{\bar{u_i^k}}^2\Big)
  = -\frac{[\omega_i^k]_0}{\enorm{\overline{g_i(x^k)}}} (u_i^k)^\T \Gamma_i u_i^k.
  \end{equation}

  \item [(iii)] $\{g_i(x^k)\}_{k\in N_3}\subset \bd^+(\K_i)$. For every $k\in N_3$, we have $\omega^k_i=0$. Also, for all such $k$, Lemma~\ref{socp:lemma-subdiffproj} item (e) implies $$V^k_i = \tau M_i \left( -1, -\frac{\overline{g_i(x^k)}}{\enorm{\overline{g_i(x^k)}}}\right),$$ for some $\tau\in [0,1]$. Then, note that $$M_i \left( -1, -\frac{\overline{g_i(x^k)}}{\enorm{\overline{g_i(x^k)}}}\right)=M_i \left( -\frac{[g_i(x^k)]_0}{\enorm{\overline{g_i(x^k)}}}, -\frac{\overline{g_i(x^k)}}{\enorm{\overline{g_i(x^k)}}}\right),$$ which means that simply multiplying \eqref{socp:sigmaborder} by $\tau$ is enough to obtain $\rho_k (u_i^k)^\T V^{k}_iu_i^k=0$ as well, since $[\omega_i^k]_0=0$.
\end{enumerate}  

  \end{enumerate}
  
 Considering exclusively any infinite subsequence indexed by $N_1$, $N_2$, or $N_3$, based on our previous analyses we observe that, for $k$ sufficiently large, \eqref{socp:akkt2ex} implies
  \[
  \liminf_{k \to \infty} \,(d^k)^\T\bigg(\nabla_x^2L(x^k,\omega^k,\mu^k)
  -\sum_{i\in I_B(x\s)} \frac{[\omega_i^k]_0}{\enorm{\overline{g_i(x^k)}}}Dg_i(x^k)^\T \Gamma_i Dg_i(x^k)\bigg)d^k\geq 0.
  \]
  Since $x^k\to x\s$, $\omega^k \to \omega\s$, $\mu^k \to \mu\s$,
  $d^k\to d\in S(x\s)$ and $\enorm{\overline{g_i(x^k)}} \to
  \enorm{\overline{g_i(x\s)}} = [g_i(x\s)]_0$ when $i \in I_B(x\s)$, we conclude that
  \[
  d^\T\bigg(\nabla_x^2L(x\s,\omega\s,\mu\s)+\sum_{i\in I_{BB}(x\s)}\sigma_i(x\s,\omega\s)\bigg)d \geq 0 
  \quad \mbox{for all } d\in S(x\s),
  \]
  where $\sigma_i(x\s,\omega\s)$ is defined as in \eqref{socp:sigma}. Therefore, $x\s$
  satisfies WSOC.   
\end{proof}


Note that Theorem~\ref{socp:minwsoc} contains a proof for the fact that every feasible limit point of any sequence $\seq{x}$ generated by an \textit{external penalty method} must satisfy WSOC if it satisfies Robinson's CQ and WCR. Moreover, with minor adaptations, it is possible to prove that the same holds for every feasible limit point of a modified extension of the \textit{augmented Lagrangian method} for NLP considered in~\cite{Birgin2016a}. And finally, we remark that if $m_1=\ldots=m_r=1$, that is, if~\eqref{NSOCP} reduces to a NLP problem, then Theorem~\ref{socp:minwsoc} recovers a result by Andreani et al.~\cite[Crlr. 4.2 and Crlr. 4.3]{akkt2} with an alternative proof. 

\section{Semidefinite programming}\label{sec:sdp}

In this section, $\sym$ is the linear space of all $m\times m$ symmetric matrices with real entries,  equipped with the (\textit{Frobenius}) inner product given by $\langle M,N \rangle\doteq \textnormal{trace}(MN)$ and the norm $\fnorm{M}\doteq \sqrt{\langle M,M\rangle}$, for every $M,N\in \sym$. We define $M\odot N$ as the (\textit{Hadamard}) entry-wise product between $M$ and $N$. Also, the cone of all symmetric positive semidefinite matrices is denoted by $\sym_+$ and $\succeq$ is the partial order induced by it, that is, $M\succeq N$ if, and only if, $M-N\in \sym_+$. Similarly, $M\succ N$ when $M-N\in \interior(\sym_+)$.

Recall that every $M\in \sym$ has a \textit{spectral decomposition} in the form $M=U\Lambda U^\T $, where $U$ is an orthogonal matrix whose columns are eigenvectors of $M$ and $\Lambda=\textnormal{Diag}(\lambda_1^U(M),\dots,\lambda_m^U(M))$ is a diagonal matrix whose entries are the eigenvalues of $M$ respective to the columns of $U$. It is well-known that the orthogonal projection of $M$ onto $\sym_+$ under $\fnorm{\cdot}$ is given by $$\projs(M)\doteq U \textnormal{Diag}(\max\{0,\lambda_1^U(M)\},\dots,\max\{0,\lambda_m^U(M)\}) U^\T.$$ The specialization of (\ref{NCP}) to an NSDP is obtained by setting $\E=\sym$ and $\K=\sym_+$, and it is often stated in the form

\begin{equation}
  \tag{NSDP}
  \begin{aligned}
    & \underset{x \in \mathbb{R}^{n}}{\text{Minimize}}
    & & f(x), \\
    & \text{subject to}
    & & g(x) \succeq 0, \\
    & & & h(x) = 0.
  \end{aligned}
  \label{NSDP}
\end{equation}

Here, for simplicity, we consider a single conic constraint since it is enough to cover all major aspects of the problem and the notation would be unnecessarily heavy otherwise. Similarly to the NSOCP case, several concepts of general conic programming can be specialized and explicitly characterized here, for example, the tangent cone to $\sym_+$ at some $M\in \sym_+$ can be written as $$T_{\mathbb{S}^m_+}(M)=\{E\in \mathbb{S}^m\colon  V^\T  E V\in \mathbb{S}^{\vert\beta\vert}_+\},$$ where $V\in \R^{m\times \vert\beta\vert}$ is any matrix with orthonormal columns that form a basis for $\textnormal{Ker}(M)$ and $\vert\beta\vert$ is its dimension (see~\cite{Shapiro1997} for details). 

Let $x$ be a feasible point of (\ref{NSDP}). In this section we always consider spectral decompositions of $g(x)$ that keep zero and nonzero eigenvalues separated, for example, $$g(x)=U\begin{bmatrix}
\Lambda & 0 \\ 
0 & 0 \\ \end{bmatrix}U^\T ,$$ where $\mathbb{S}^{\alpha}\ni\Lambda\succ 0$ and $\alpha\doteq \alpha(x)$ is the set of indices of the positive eigenvalues of $g(x)$. Let $\beta\doteq \beta(x)$ be the set of indices of the null eigenvalues of $g(x)$ and partition the columns $U$ with respect to $\alpha$ and $\beta$ as follows: $U\doteq [U_\alpha, U_\beta]$. For every $d\in\R^n$, define $D\tilde{g}(x)[d]\doteq U^\T Dg(x)[d] U$ as a reverse conjugation of $Dg(x)[d]$ around $g(x)$ and set $$D\tilde{g}(x)[d]=\begin{bmatrix}
(D\tilde{g}(x)[d])_{\alpha\alpha} & (D\tilde{g}(x)[d])_{\alpha\beta} \\ 
(D\tilde{g}(x)[d])_{\alpha\beta}^\T  & (D\tilde{g}(x)[d])_{\beta\beta} \\ 
\end{bmatrix}$$ as a partition of $D\tilde{g}(x)[d]$ with respect to $\alpha$ and $\beta$. Note that since $U$ is an orthogonal matrix, the inner product is invariant to reverse conjugation in terms of $U$, that is, $$\langle A,B\rangle=\textnormal{trace}(AB)=\textnormal{trace}(UU^\T AUU^\T B)=\textnormal{trace}(U^\T AUU^\T BU)=\langle \tilde{A},\tilde{B}\rangle$$ for all $A,B\in \sym$.

The critical cone of \eqref{NSDP} at a feasible point $x\s$ is given by $$C(x\s)= \{d\in \R^n\colon  \nabla f(x\s)^\T d=0, Dh(x\s)d=0,(D\tilde{g}(x\s)[d])_{\beta\beta}\succeq 0\}.$$ Under Robinson's CQ, if $x\s$ is a KKT point associated with some Lagrange multipliers $\omega\s\in \sym_+$ and $\mu\s\in \R^p$ that satisfy strict complementarity, then the critical cone becomes equal to the critical subspace $$S(x\s)=\{d\in \R^n\colon  Dh(x\s)d=0, (D\tilde{g}(x\s)[d])_{\beta\beta}=0\}$$ since $\tilde{\omega}\s_{\alpha\alpha}=0$ and  $\nabla f(x\s)^\T d = \langle Dg(x\s)[d],\omega\s \rangle - \langle Dh(x\s)d, \mu\s\rangle=0$ for every $d\in \R^n$.

In~\cite{ahv}, Andreani, Haeser, and Viana proposed an extension of the AKKT condition from NLP to \eqref{NSDP} as well. We state it as follows:

\begin{definition}[AKKT for NSDP]\label{sdp:akkt}
A feasible point $x\s$ of \eqref{NSDP} satisfies the AKKT condition when there are sequences $\seq{x}\to x\s$, $\seq{\omega}\subset \sym_+$, and $\seq{\mu}\subset \R^p$ such that 
\begin{equation}\label{sdp:akktstat}
\nabla_x L(x^k,\omega^k,\mu^k)\to 0
\end{equation}
and
\begin{equation}\label{sdp:akktcomp}
\lambda_i^U(g(x\s))>0\Rightarrow \lambda_i^{S^k}(\omega^k)= 0,
\end{equation}
for sufficiently large $k$, where $U$ diagonalizes $g(x\s)$, $S^k$ diagonalizes $\omega^k$ for each $k$, and $S^k\to U$.
\end{definition}

If $x\s$ is a KKT point of (\ref{NSDP}) associated with multipliers $\omega\s\in \sym_+$ and $\mu\s\in \R^p$, note that the complementarity condition $\langle g(x\s),\omega\s \rangle=0$ holds for $g(x\s),\omega\s\in\sym_+$ if, and only if, $g(x\s)\omega\s=0$, then it is elementary to check that $g(x\s)$ and $\omega\s$ must be simultaneously diagonalizable (i.e. they commute) in this case. In light of this, note that Definition~\ref{sdp:akkt} relaxes the commutativity between $g(x\s)$ and $\omega\s$ by requiring $S^k\to U$. 

Also in~\cite{ahv}, the authors prove that AKKT as in Definition~\ref{sdp:akkt} is a necessary optimality condition, independently of the fulfilment of constraint qualifications. We state it below in the same form as Theorem~\ref{socp:minakkt}, with some emphasis on how the sequences that compose it are generated.

\begin{theorem}\label{sdp:minakkt}
Let $x\s$ be a local minimizer of \eqref{NSDP}. Then, for any sequence $\{\rho_k\}_{k\in \N}\to +\infty$, there exists some $\seq{x}\to x\s$, such that for every $k$, $x^k$ is a local minimizer of the regularized penalty function 
$$F_k(x) \doteq   f(x)+\frac{1}{4}\enorm{x-x\s}^4 
  + \frac{\rho_k}{2} \left(\fnorm{\Pi_{\sym_+}(-g(x))}^2
  + \enorm{h(x)}^2 \right).$$
Also, the multiplier sequences given by $\omega^k \doteq \rho_k \Pi_{\sym_+}(-g(x^k))$ and $\mu^k \doteq   \rho_k h(x^k)$ satisfy \eqref{sdp:akktstat} and \eqref{sdp:akktcomp} with $\seq{x}$. Consequently, since $\omega^k$ and $g(x^k)$ are simultaneously diagonalizable in this case for every $k\in \N$, $x\s$ satisfies AKKT.
\end{theorem}

Under Robinson's CQ, the sequences $\seq{\omega}$ and $\seq{\mu}$ are bounded, and also all limit points of those sequences are Lagrange multipliers associated with $x\s$~\cite[Thm. 6.1]{ahv}. That is, AKKT implies KKT in this case. An augmented Lagrangian algorithm is also presented in~\cite{ahv} for NSDP, whose global convergence theory is built around AKKT. Such results were sharpened in~\cite{Santos2021} and further extended in~\cite{Andreani2020}, for the general \eqref{NCP}.

\subsection{Second-order analysis}

As mentioned before, there are many different works that deal with a specialized second-order analysis for \eqref{NSDP}, which mainly differ in the assumptions required for it and the techniques employed to characterize the sigma-term. As far as we know, the first work on this topic is due to Shapiro~\cite[Sec. 4]{Shapiro1997}, who obtained the very useful and practical expression
\begin{equation}\label{sdp:sigmaterm} 
\sigma(x\s,\omega\s)= [2\langle \omega\s,\partial_i g(x\s) g(x\s)^{\dag} \partial_j g(x\s) \rangle]_{i,j\in\{1,\dots,n\}},
\end{equation}
where $g(x\s)^\dag$ is the \textit{Moore-Penrose pseudoinverse} of $g(x\s)$. Shapiro's idea was to write the semidefinite cone using the second-order directional derivative of the least eigenvalue function $\lambda_{\min}$, as follows:
\[
	\sym_+= \{M\in \sym\colon \lambda_{\min}(M)\geqslant 0\},
\] 
and the expression of the sigma-term comes from the expression of the second-order directional derivative of $\lambda_{\min}$. Moreover, his second-order analysis was based on the uniqueness of the Lagrange multiplier, via nondegeneracy, and strict complementarity (Theorem \ref{classicwsoc}). Then, Jarre~\cite[Thm. 2]{Jarre2012} presented another way of achieving Shapiro's characterization of the sigma-term, and consequently an alternative proof for Theorem~\ref{classicwsoc}, using a locally equivalent formulation of \eqref{NSDP} based on the Schur complement of $g(x\s)_{\alpha\alpha}$, which turned out to be a more elementary proof. Later, Lourenço, Fukuda, and Fukushima~\cite[Props. 5.1 and 5.2]{Lourenco2018}, studied a characterization of the semidefinite cone with squared slack variables
\[
	\sym_+= \{M\in \sym\colon \exists Z\in \sym, \ M-ZZ=0\},
\] 
which induces a reformulation of~\eqref{NSDP} as a NLP problem. Then, the authors related the classical second-order conditions for NLP with the second-order conditions for~\eqref{NSDP} (with the curvature term), under the same hypotheses as Shapiro and Jarre. Forsgren~\cite[Thm. 2]{Forsgren2000}, on the other hand, proved that strict complementarity was not needed for WSOC when assuming a different notion of regularity that treats structural sparsity and also uses Schur complements. In this section, we use the characterization
\[
	\sym_+= \{M\in \sym\colon \vert\vert\Pi_{\sym_+}(-M)\vert\vert^2=0\},
\] 
and the generalized derivative of the orthogonal projection, to obtain second-order results that do not require uniqueness of multipliers, nor strict complementarity.

Our approach is based on extracting second-order information from AKKT and Theorem \ref{sdp:minakkt} and we do this in a similar manner of the previous section, which means we begin by exhibiting a characterization of the derivative of $\Pi_{\sym_+}$, then we extend the WCR condition from NLP to NSDP and, at last, we compute the sigma-term using the second-order (generalized) derivative of $F_k$. 

Based on the works of Bonnans at al.~\cite{FredericBonnans1998} and Pang et al.~\cite{PSS03}, Sun~\cite{Sun2006} characterized the $B$-subdifferential of the projection onto the semidefinite cone. To make a proper reference, we define the following matrix:

$$\mathcal{B}(\lambda^U(M))\doteq \left[\frac{\max\{\lambda_i^U(M),0\}+\max\{\lambda_j^U(M),0\}}{\vert\lambda_i^U(M)\vert+\vert\lambda_j^U(M)\vert}\right]_{i,j\in\{1,\dots,m\}},$$ where $0/0$ is set as $1$ and $U$ is an orthogonal matrix that diagonalizes $M$. Next, we make a slightly adapted transcription of a proposition by Qi~\cite[Prop. 2.5]{Qi2009} summarizing Sun's result:

\begin{proposition}\label{sdp:propqi}
Suppose that $M=U\Lambda U^\T $ is the spectral decomposition of $M\in \mathbb{S}^m$ and let $\alpha,\beta$ and $\gamma$ be the sets of indices of the positive, zero and negative eigenvalues of $M$, respectively. Without loss of generality, assume those three blocks are separated and that $U\doteq [U_\alpha, U_\beta,U_\gamma]$. Then, for any $V\in \partial_B \Pi_{\mathbb{S}^m_+} (-M)$ there exists a $V_{\vert\beta\vert}\in \partial_B \Pi_{\mathbb{S}^{\vert\beta\vert}_+}(0)$ such that 
\begin{equation}\label{qi}
V[H]=U
\begin{bmatrix}
0 & 0 & \tilde{H}_{\alpha\gamma} \odot \mathcal{B}(\lambda^U(M))_{\alpha\gamma} \\ 
0  & V_{\vert\beta\vert}[\tilde{H}_{\beta\beta}] & \tilde{H}_{\beta\gamma} \\ 
\tilde{H}_{\alpha\gamma}^\T  \odot \mathcal{B}(\lambda^U(M))_{\alpha\gamma}^\T  & \tilde{H}_{\beta\gamma}^\T & \tilde{H}_{\gamma\gamma} \\
\end{bmatrix} U^\T 
\end{equation}
for every $H\in \mathbb{S}^m$, where $\tilde{H}\doteq U^\T H U$. Conversely, for every $V_{\vert\beta\vert}\in \partial_B \Pi_{\mathbb{S}^{\vert\beta\vert}_+}(0)$, there exists some $V\in \partial_B \Pi_{\mathbb{S}^m_+} (-M)$ such that (\ref{qi}) holds. 
\end{proposition}

Even though we assume the eigenvalues are separated by sign, the ordering inside each partition is not relevant.  Note that Proposition~\ref{sdp:propqi} is still true if we replace the B-subdifferential for the Clarke subdifferential.

\begin{corollary}\label{cor:qi}
Under the hypotheses of Proposition~\ref{sdp:propqi}, for any $V\in \partial \Pi_{\mathbb{S}^m_+} (-M)$ there exists a $V_{\vert\beta\vert}\in \partial \Pi_{\mathbb{S}^{\vert\beta\vert}_+}(0)$ such that \eqref{qi} holds. Conversely, for every $V_{\vert\beta\vert}\in \partial \Pi_{\mathbb{S}^{\vert\beta\vert}_+}(0)$, there exists some $V\in \partial \Pi_{\mathbb{S}^m_+} (-M)$ such that (\ref{qi}) holds.
\end{corollary}

\begin{proof}
Let $V\in \partial \Pi_{\mathbb{S}^m_+} (-M)$. Then, $V=\sum_{i=1}^s a_i V^i$, for some $s\in \N$, some $a_i\geqslant 0$, and some $V^i\in \partial_B \Pi_{\mathbb{S}^m_+} (-M)$, $i\in \{1,\ldots,s\}$, with $\sum_{i=1}^s a_i=1$. This means there are $V^i_{\vert\beta\vert}\in \partial_B \Pi_{\mathbb{S}^{\vert\beta\vert}_+}(0)$, $i\in \{1,\ldots,s\}$, such that \eqref{qi} holds. Hence, for every $H\in \sym$, we have $V[H]=\sum_{i=1}^s a_i V^i[H]$ and the proof is over, because $\sum_{i=1}^s a_i V^i_{\vert\beta\vert}[\tilde{H}_{\beta\beta}]\in \partial \Pi_{\mathbb{S}^{\vert\beta\vert}_+}(0)$. The converse is analogous.   
\end{proof}

In order to study perturbations of the critical subspace around a given point $x\s$ via WCR, let $\bar{\alpha}$ and $\bar{\beta}$ represent the indices of positive and zero eigenvalues of $g(x\s)$, respectively, regarding the decomposition 
\begin{equation}\label{sdp:decomp}
g(x\s)=U \textnormal{Diag}(\lambda^U(g(x\s)))U^\T,
\end{equation}
where $U\doteq [U_{\bar\alpha},U_{\bar\beta}]$ is a matrix whose columns are eigenvectors of $g(x\s)$ and, in particular, the columns of $U_{\bar\beta}$ form a basis for $\textnormal{Ker}(g(x\s))$. Moreover, we will use a construction from Bonnans and Shapiro's book~\cite[Ex. 3.98 and Ex. 3.140]{bshapiro}, which will be stated as a lemma below:

\begin{lemma}\label{sdp:lemma-bs}
	Let $M\s\in \sym_+$, set $\bar{\beta}$ as the indices of zero eigenvalues of $M$, and let $U_{\bar{\beta}}$ be a matrix with orthonormal columns that span $\textnormal{Ker}(M\s)$. There exists a neighborhood $\mathcal{N}$ of $M\s$ and an analytic matrix function $\mathcal{U}_{\bar{\beta}}\colon \mathcal{N}\to \R^{m\times \vert\bar\beta\vert}$ such that $\mathcal{U}_{\bar{\beta}}(M\s)=U_{\bar{\beta}}$ and, for every $M\in \mathcal{N}$, the columns of $\mathcal{U}_{\bar{\beta}}(M)$ form an orthonormal basis for the space spanned by the eigenvectors associated with the $\vert\bar\beta\vert$ smallest eigenvalues of $M$.
\end{lemma}

This construction allows us to approximate the critical subspace around $x\s$. Indeed, let $\mathcal{N}$ be the neighborhood of $g(x\s)$ and $\mathcal{U}_{\bar{\beta}}\colon \mathcal{N}\to \R^{m\times \vert\bar\beta\vert}$ be the function given by Lemma~\ref{sdp:lemma-bs} such that $\mathcal{U}_{\bar{\beta}}(g(x\s))=U_{\bar\beta}$. Then, for every $x$ close enough to $x\s$ so that $g(x)\in \mathcal{N}$, consider the following set: $$S(x,x\s)=\{d\in \R^n\colon \mathcal{U}_{\bar{\beta}}(g(x))^\T Dg(x)[d]\mathcal{U}_{\bar{\beta}}(g(x))=0, Dh(x)d=0\},$$
which will be called \textit{perturbed critical subspace} at $x$, centered at $x\s$.

Extending WCR from NLP to \eqref{NSDP} is not a trivial task because the notion of ``rank'' of the three-dimensional tensor $Dg(x)$ may have multiple meanings. Fortunately, there is a useful characterization of nondegeneracy by Shapiro and Fan \cite{shapfan}, which provides some insight on how to talk about rank in NSDP. Next, we make a transcription of this result as stated in \cite[Prop. 6]{Shapiro1997}, for completeness.

\begin{proposition}\label{prop:ndgshapiro}
Suppose that the dimension of $\textnormal{Ker}(g(x\s))$ is $\vert\bar\beta\vert$ and let $U_{\bar\beta}\doteq[u_1,\dots,u_{\vert\bar\beta\vert}]\in \R^{m\times \vert\bar\beta\vert}$ be a matrix whose columns form a basis for $\textnormal{Ker}(g(x\s))$. Then, nondegeneracy holds at a feasible point $x\s$ of \eqref{NSDP} if, and only if, the set of $n$-dimensional vectors $\{v_{ij}\colon 1\leqslant i\leqslant j\leqslant \vert\bar\beta\vert\}\cup\{\nabla h_i(x\s)\colon i\in \{1,\dots,p\}\}$ is linearly independent, where $v_{ij}\doteq [u_i^\T \partial_\ell g(x\s) u_j]_{\ell\in\{1,\dots,n\}}$.
\end{proposition}

Inspired by this characterization, we define WCR as follows:

\begin{definition}[WCR for NSDP]\label{sdp:wcr}
Let $x\s$ be a feasible point of \eqref{NSDP} and let \eqref{sdp:decomp} be a spectral decomposition of $g(x\s)$. We say that $x\s$ satisfies the weak constant rank (WCR) property when there exists a neightborhood $\mathcal{N}$ of $x\s$ such that the set $$\{\bar v_{ij}(x)\colon 1\leqslant i\leqslant j\leqslant \vert\bar\beta\vert\}\cup \{\nabla h_i(x)\colon i\in \{1,\dots,p\}\}$$ has the same rank for every $x\in \mathcal{N}$, where
\[
	\bar v_{ij}(x)\doteq [\bar{u}_i(x)^\T \partial_\ell g(x) \bar{u}_j(x)]_{\ell\in\{1,\dots,n\}}
\]
and $\bar{u}_1(x),\dots,\bar{u}_{\vert\bar\beta\vert}(x)\in \R^{m}$ denote the columns of $\mathcal{U}_{\bar{\beta}}(g(x))$.
\end{definition}

Also, in the following lemma we prove that WCR as in Definition \ref{sdp:wcr} is equivalent to the inner semicontinuity of the mapping $x\mapsto S(x,x\s)$ at $x\s$.

\begin{lemma}
A feasible point $x\s$ satisfies WCR if, and only if, the set-valued mapping $x\mapsto S(x,x\s)$ is inner semicontinuous at $x\s$.
\end{lemma}

\begin{proof}

First, we shall prove that, for every $x\in \R^n$,
\begin{equation}\label{sdp:subspacechar}
S(x,x\s)= \left\{d\in \R^n\colon \bar v_{ij}(x)^\T d=0, \ \si{\forall i,j \text{ such that }} 1\leqslant i\leqslant j\leqslant \vert\bar\beta\vert; \ Dh(x)d=0\right\}.
\end{equation}
Note that for each $\ell\in \{1,\dots,n\}$ we have
\begin{equation*}
\begin{aligned}
\mathcal{U}_{\bar{\beta}}(g(x\s))^\T \partial_\ell g(x)\mathcal{U}_{\bar{\beta}}(g(x\s))
& = 
\begin{bmatrix}
\bar{u}_{1}(x)^\T \partial_\ell g(x) \bar{u}_{1}(x) & \cdots & \bar{u}_{1}(x)^\T \partial_\ell g(x) \bar{u}_{\vert\bar\beta\vert}(x) \\
   \vdots    & \ddots & \vdots \\
\bar{u}_{\vert\bar\beta\vert}(x)^\T \partial_\ell g(x) \bar{u}_{1}(x) & \cdots & \bar{u}_{\vert\bar\beta\vert}(x)^\T \partial_\ell g(x) \bar{u}_{\vert\bar\beta\vert}(x)
\end{bmatrix},\\
\end{aligned}
\end{equation*}
hence, considering that $\bar{u}_{j}(x)^\T \partial_\ell g(x) \bar{u}_{i}(x)=\bar{u}_{i}(x)^\T \partial_\ell g(x) \bar{u}_{j}(x)=(\bar{v}_{ij}(x))_\ell$ for every $i,j,\ell$, we obtain $$\mathcal{U}_{\bar{\beta}}(g(x\s))^\T Dg(x)[d]\mathcal{U}_{\bar{\beta}}(g(x\s))=\si{\sum_{\ell=1}^n \mathcal{U}_{\bar{\beta}}(g(x\s))^\T \partial_\ell g(x)\mathcal{U}_{\bar{\beta}}(g(x\s))d_\ell=}
\begin{bmatrix}
\bar{v}_{11}(x)^\T d & \cdots & \bar{v}_{1\vert\bar\beta\vert}(x)^\T d \\
   \vdots    & \ddots & \vdots \\
\bar{v}_{1\vert\bar\beta\vert}(x)^\T d & \cdots &\bar{v}_{\vert\bar\beta\vert\vert\bar\beta\vert}(x)^\T d
\end{bmatrix},$$ whence follows \eqref{sdp:subspacechar}.

Now, similarly to the NSOCP case, since $\partial_\ell g$ and $\mathcal{U}_{\bar\beta}$ are continuous, $\bar{v}_{ij}$ is also continuous, then a result from Facchinei and Pang \cite[Prop. 3.2.9]{fachpang} tells us that WCR is equivalent to the outer semicontinuity of the mapping $x\mapsto S(x,x\s)\pol$ at $x\s$, where $$S(x,x\s)\pol=\left\{\sum_{1\leqslant i\leqslant j\leqslant \vert\bar\beta\vert} a_{ij}\bar{v}_{ij}(x) + \sum_{i=1}^p b_i \nabla h_i(x) \ \colon \ a_{ij}\in \R, b_i\in \R\right\}$$ using the characterization in \eqref{sdp:subspacechar}. Then, the desired result follows from \cite[Thm. 1.1.8]{aubinf}, which states that the inner semicontinuity of a set-valued mapping at a given point is equivalent to the outer semicontinuity of its polar at that point.   
\end{proof}

Clearly, WCR as in Definition \ref{sdp:wcr} is implied by nondegeneracy, in view of Proposition \ref{prop:ndgshapiro}. Also, let us assume for a moment that $g(x)$ is a structurally diagonal matrix constraint whose diagonal elements are denoted by $g_1(x),\ldots,g_m(x)$, and let $x\s$ be such that $g(x)\in \sym_+$. Without loss of generality, let us assume that $g_1(x\s)>\ldots>g_{\vert\bar\alpha\vert}(x\s)>g_{\vert\bar\alpha\vert+1}(x\s)=\ldots=g_{\vert\bar\alpha\vert+\vert\bar\beta\vert}(x\s)=0$. Then, we can take
\[
	\mathcal{U}_{\bar\beta}(g(x))\doteq \begin{bmatrix}
		0 \\
		\I_{\vert\bar\beta\vert}
	\end{bmatrix}
\]
as a constant function to obtain that $\bar{v}_{ii}(x)=\nabla g_i(x)$ for every $i\in \{1,\ldots,\vert\bar\beta\vert\}$ and $\bar{v}_{ij}(x)=0$ when $i\neq j$. That is, the WCR condition as in Definition~\ref{sdp:wcr} recovers the NLP definition of WCR when such NLP constraints are modelled as a single structurally diagonal matrix constraint, with this choice of $\mathcal{U}_{\bar\beta}$. It is important to keep in mind, however, that even if the constraints $g_1(x)\geqslant 0,\ldots,g_m(x)\geqslant 0$ satisfy LICQ at $x\s$, nondegeneracy may not hold at $x\s$, as observed by Shapiro in~\cite[p. 309]{Shapiro1997}. The converse, on the other hand, is true. Now, recall that \cite[Ex. 5.2]{ams07} exhibits an NLP problem with a feasible point that satisfies ``MFCQ+WCR'', but not LICQ, and the above discussion tells us that it can be used again to prove that nondegeneracy is strictly stronger than ``Robinson's CQ+WCR''. Moreover, nondegeneracy and LICQ are equivalent when considering multiple unidimensional semidefinite constraints, and so are Definition~\ref{sdp:wcr} and the NLP version of WCR. Furthermore, Forsgren~\cite[Sec. 2.3]{Forsgren2000} and Andreani et al.~\cite[Def. 3.2]{weak-sparse-cq} considered regularity notions different from nondegeneracy, that also recover the standard LICQ in NLP. Thus, in all cases, regardless of modelling, the example of~\cite[Ex. 5.2]{ams07} can be used to conclude that ``Robinson's CQ+WCR'' is strictly weaker than all existing notions of nondegeneracy.

With this in mind, we proceed to the main result of this section:

\begin{theorem}\label{sdp:minwsonc}
If $x\s$ is a local minimizer such that Robinson's CQ  and the WCR property hold, then there are some Lagrange multipliers $\omega\s\in \sym_+$ and $\mu\s\in \R^p$ such that the KKT conditions and WSOC hold for this pair of multipliers.
\end{theorem}

\begin{proof}
If $x\s$ is a local minimizer of \eqref{NSDP}, Theorem~\ref{sdp:minakkt} tells us that for any given $\{\rho_k\}_{k\in \N}\to +\infty$, there is some sequence $\seq{x}\to x\s$ such that, for every $k\in \N$, $x^k$ is a local minimizer of the penalty function $$F_k(x) =  f(x)+\frac{1}{4}\enorm{x-x\s}^4 
  + \frac{\rho_k}{2} \left(\fnorm{\Pi_{\sym_+}(-g(x))}^2
  + \enorm{h(x)}^2 \right).$$
Hence, it satisfies the first-order stationarity condition 
\begin{equation*}
\si{\nabla F(x^k)=\nabla f(x^k)+\enorm{x^k-x\s}^2(x^k-x\s) + Dh(x^k)^\T (\rho_k h(x^k) )- Dg(x^k)^*[\rho_k\Pi_{\mathbb{S}_+^m}(-g(x^k))]=0.}
\jo{\begin{split}
\nabla F(x^k)=\nabla f(x^k)+\enorm{x^k-x\s}^2(x^k-x\s) + Dh(x^k)^\T (\rho_k h(x^k) )- \hspace{1cm}\\ - Dg(x^k)^*[\rho_k\Pi_{\mathbb{S}_+^m}(-g(x^k))]=0.
\end{split}}
\end{equation*}
Setting the approximate multipliers $\omega^k\doteq \rho_k  \Pi_{\mathbb{S}_+^m}(-g(x^k))$ and $\mu^k\doteq \rho_k  h(x^k)$, we obtain \eqref{sdp:akktstat} and \eqref{sdp:akktcomp} due to Theorem~\ref{sdp:minakkt}. Also, $x^k$ is second-order stationary in the nonsmooth sense (see Theorem~\ref{thm:nonsmooth2nd}), which means that, for each unitary vector $d\in \R^n$, there exists some $\chi^k\in \partial (\Pi_{\mathbb{S}^m_+} \circ -g)(x^k)$ such that
\begin{equation}\label{sdp:akkt2ex}
\si{\begin{aligned}
d^\T\nabla^2 F(x^k)d
= & \ d^\T\left(\nabla^2 f(x^k)-D^2g(x^k)^*[\rho_k \Pi_{\mathbb{S}_+^m}(-g(x^k))]+\sum_{i=1}^p \mu_i^k\nabla^2 h(x^k)\right)d \ +\\
& \ +\rho_k (Dh(x^k)d)^\T Dh(x^k)d - d^\T\left( Dg(x^k)^*\left[\rho_k\chi^k[d]\right] \right) + d^\T\Delta^k d\geqslant \ 0
\end{aligned}}
\jo{\begin{aligned}
& \ d^\T\nabla^2 F(x^k)d \\
= & \ d^\T\left(\nabla^2 f(x^k)-D^2g(x^k)^*[\rho_k \Pi_{\mathbb{S}_+^m}(-g(x^k))]+\sum_{i=1}^p \mu_i^k\nabla^2 h(x^k)\right)d \ +\\
& \ +\rho_k (Dh(x^k)d)^\T Dh(x^k)d - d^\T\left( Dg(x^k)^*\left[\rho_k\chi^k[d]\right] \right) + d^\T\Delta^k d\\
\geqslant & \ 0
\end{aligned}}
\end{equation}
where $\Delta^k\doteq  \enorm{x^k-x\s}^2\I_n + 2(x^k-x\s)(x^k-x\s)^\T$ and $\nabla^2 F(x^k)$ denotes the element of $\partial^2 F_k(x^k)$ that is defined in terms of $\chi^k$, as an abuse of notation.
By Theorem~\ref{thm:chainrule}, there exists some $V^k\in \partial \Pi_{\mathbb{S}^m_+} (-g(x^k)),$ such that \jo{$}$\chi^k=V^k\circ -Dg(x^k),$\jo{$} for every $k\in \N$.

Under Robinson's CQ, the sequences $\seq{\omega}$ and $\seq{\mu}$ are bounded, so they have convergent subsequences which we will consider to be themselves from now on, without loss of generality. Denote their limits by $\omega\s$ and $\mu\s$, respectively. In~\cite[Thm. 6.1]{ahv}, the authors also prove that $\omega\s$ and $\mu\s$ are Lagrange multipliers associated with $x\s$. 

Now, let $d\in S(x\s)$. By WCR there is a sequence $\{d^k\}_{k\in \N}\to  d$ such that $d^k\in S(x^k,x\s)$ for every $k$. Rewriting \eqref{sdp:akkt2ex} in terms of $d^k$, $V^k$, $\omega^k$, and $\mu^k$, we obtain
\begin{equation}\label{sdp:akkt2re}
\si{(d^k)^\T \nabla^2_x L(x^k,\omega^k,\mu^k)d^k+\rho_k (Dh(x^k)d^k)^\T Dh(x^k)d^k+\rho_k\left\langle Dg(x^k)[d^k],V^k\left[Dg(x^k)[d^k]\right]\right\rangle\geqslant -\delta^k,}
\jo{\begin{split}(d^k)^\T \nabla^2_x L(x^k,\omega^k,\mu^k)d^k+\rho_k (Dh(x^k)d^k)^\T Dh(x^k)d^k+\hspace{2cm}\\+\rho_k\left\langle Dg(x^k)[d^k],V^k\left[Dg(x^k)[d^k]\right]\right\rangle\geqslant -\delta^k,\end{split}}
\end{equation}
where $\delta^k\doteq(d^k)^\T\Delta^kd^k\to 0$. The following paragraphs prove that \eqref{sdp:akkt2re} implies
\begin{equation}\label{sdp:wsoncre}
d^\T \nabla^2_x L(x\s,\omega\s,\mu\s)d+2\left\langle Dg(x\s)[d], \omega\s Dg(x\s)[d] g(x\s)^{\dag}\right\rangle\geqslant 0,
\end{equation}
which is enough to complete the proof since $$2\left\langle \omega\s, Dg(x\s)[d]  g(x\s)^{\dag} Dg(x\s)[d] \right\rangle=d^\T\sigma(x\s,\omega\s)d,$$ for every $d\in \R^n$ due to \eqref{sdp:sigmaterm}.

To complete that task, we\jo{ proceed to } analyse the behaviour of the sequence $\{\rho_k\langle Dg(x^k)[d^k],V^k[Dg(x^k)[d^k]]\rangle\}_{k\in\N}$ in distinct cases. In the following paragraphs, we let $\alpha\doteq \alpha(x^k)$, $\beta\doteq \beta(x^k)$, and $\gamma\doteq \gamma(x^k)$ be the sets of indices of the positive, zero and negative eigenvalues of $g(x^k)$, respectively, regarding the spectral decomposition $$g(x^k)=S^k \textnormal{Diag}(\lambda^{S^k}(g(x^k)))(S^k)^\T$$ with $S^k\to U$. Recall that, by construction, the columns of $\mathcal{U}_{\bar{\beta}}(g(x^k))$ span the eigenspace associated with the $\bar\beta$ smallest eigenvalues of $g(x^k)$, for all $k$ sufficiently large. Denote the submatrix of $S^k$ that has the eigenvectors associated with the $\bar\beta$ smallest eigenvectors of $g(x^k)$ in its columns by $S^k_{\bar\beta}$, and since $d^k\in S(x^k,x\s)$, we have $(S^k_{\bar\beta})^\T Dg(x^k)[d^k]S^k_{\bar\beta}=0$ for every $k$ large enough. We proceed by analysing a few cases:

\begin{enumerate}
\item If $g(x\s)\succ 0$, then $-g(x^k)\prec 0$ for $k$ sufficiently large. For such $k$, since $\gamma(x^k)=\beta(x^k)=\emptyset$ and $\alpha(x^k)=\{1,\ldots,m\}$, we obtain and $V^k[Dg(x^k)[d^k]]=0$ from Proposition~\ref{sdp:propqi}, which implies $$\rho_k\langle Dg(x^k)[d^k], \ V^k[Dg(x^k)[d^k]]\rangle=0.$$  Also, note that $\sigma(x\s,\omega\s)=0$ in this case, because $\langle g(x\s),\omega\s\rangle =0$ implies $\omega\s=0$. 

\item If $g(x\s)=0$, then $g(x\s)^{\dag}=0$ and $\sigma(x\s,\omega\s)=0$ as well. On the other hand, note that 
\begin{equation*}
\begin{array}{ll}
\langle Dg(x^k)[d^k], \ V^k[Dg(x^k)[d^k]]\rangle &=\langle D\tilde{g}(x^k)[d^k], \ (U^k)^\T V^k[Dg(x^k)[d^k]]U^k\rangle\\
& =0,
\end{array}
\end{equation*} 
because $\bar{\beta}=\{1,\ldots,m\}$ and $d^k\in S(x^k,x\s)$ implies \jo{$}$D\tilde{g}(x^k)[d^k]=(S^k)^\T Dg(x^k)[d^k]S^k=0$\jo{$} in this case.

\item If $g(x\s)\succeq 0$, but $g(x\s)\neq 0$, assume the diagonalization is taken such that nonzero eigenvalues are separated from the others and the common zero eigenvalues between $g(x\s)$ and $\omega\s$ are discriminated, that is, 

\begin{center}
$g(x\s)=U\begin{bmatrix}
\Lambda & 0 & 0\\ 
0 & 0 & 0\\
0 & 0 & 0
\end{bmatrix}
U^\T
$
\ \ and \ \
$\omega\s=U\begin{bmatrix}
0 & 0 & 0 \\
0 & 0 & 0 \\ 
0 & 0 & \Gamma\\ 
\end{bmatrix}U^\T ,$
\end{center}

where $\mathbb{S}^{\vert\bar{\alpha}\vert}\ni \Lambda\succ 0$ and $\mathbb{S}^{\vert\bar{\gamma}\vert}\ni\Gamma\succ 0$ are diagonal matrices, $\bar{\kappa}\cup\bar{\gamma}$ is a partition of $\bar{\beta}$, and $U$ is orthogonal. Denoting $H\doteq Dg(x\s)[d]$, since $d\in S(x\s)$ we get $\tilde{H}_{\bar{\beta}\bar{\beta}}=0$ and

\begin{center}
\begin{tabular}{ll}
$g(x\s)^{\dag} H \omega\s$ &$=$ 
$U
\begin{bmatrix}
\Lambda^{-1} & 0 & 0 \\ 
0 & 0 & 0 \\
0 & 0 & 0 \\
\end{bmatrix} U^\T  U
\begin{bmatrix}
\tilde{H}_{\bar{\alpha}\bar{\alpha}} & \tilde{H}_{\bar{\alpha}\bar{\kappa}} & \tilde{H}_{\bar{\alpha}\bar{\gamma}}\\ 
\tilde{H}_{\bar{\kappa}\bar{\alpha}} & 0 & 0\\
\tilde{H}_{\bar{\gamma}\bar{\alpha}} & 0 & 0\\ 
\end{bmatrix} U^\T U\begin{bmatrix}
0 & 0 & 0 \\
0 & 0 & 0 \\ 
0 & 0 & \Gamma\\ 
\end{bmatrix} U^\T$ \\[20pt]

& $=
U\begin{bmatrix}
0 & 0 & \Lambda^{-1}\tilde{H}_{\bar{\alpha}\bar{\gamma}}\Gamma \\ 
0 & 0 & 0\\
0 & 0 & 0\\
\end{bmatrix}U^\T.$
\end{tabular}
\end{center}

Conveniently, 
\begin{equation}
\Lambda^{-1}\tilde{H}_{\bar{\alpha}\bar{\gamma}}\Gamma=\left[\frac{\lambda^U_{j}(\omega\s)}{\lambda^U_i(g(x\s))} \tilde{H}_{ij}\right]_{i\in \overline{\alpha},j\in\overline{\gamma}}\doteq A\odot \tilde{H}_{\bar{\alpha}\bar{\gamma}},\\[5pt]
\end{equation}
where $\mathbb{R}^{\vert\overline{\alpha}\vert\times\vert\overline{\gamma}\vert}\ni A\doteq \left[\lambda^U_j(\omega\s)\lambda^U_i(g(x\s))^{-1}\right]_{i\in \overline{\alpha},j\in\overline{\gamma}}$. Also, note that

$$\langle H,g(x\s)^{\dag} H \omega\s \rangle=\left\langle\tilde{H}, \frac{1}{2}\begin{bmatrix}
0 & 0 & A\odot\tilde{H}_{\bar{\alpha}\bar{\gamma}} \\ 
0 & 0 & 0 \\
(A\odot\tilde{H}_{\bar{\alpha}\bar{\gamma}})^\T & 0 & 0 \\ 
\end{bmatrix}\right\rangle.$$

In view of this characterization of the sigma-term over $d$, its relation with \eqref{sdp:akkt2re} can be made explicit. Consider the following spectral decomposition of $g(x^k)$:

\begin{equation}\label{sdp:decgk}
g(x^k)=S^k\begin{bmatrix}
\Lambda^k_+ & 0 & 0 & 0\\
0 & \Lambda^k_- & 0 & 0\\
0 & 0 & 0 & 0\\  
0 & 0 & 0 & \Psi^k \\
\end{bmatrix}
(S^k)^\T,
\end{equation}
where we separate the eigenvalues of $g(x^k)$ primarily by their sign and, secondarily, by their limit points. For instance, $\Lambda^k_+\in \mathbb{S}^{\vert\alpha_+\vert}$ are the positive ones that converge to $\Lambda$, while $\Lambda^k_-\in \mathbb{S}^{\vert\alpha_-\vert}$ are the positive ones that converge to zero. The squared block of zeros in the diagonal of \eqref{sdp:decgk} is of dimension $\beta$ and $\Psi^k\in \mathbb{S}^{\vert\gamma\vert}$ contains the negative eigenvalues of $g(x^k)$. Also, $\vert\alpha_+\vert + \vert\alpha_-\vert + \vert\beta\vert + \vert\gamma\vert=m$. Recall that $S^k$ simultaneously diagonalizes $g(x^k)$ and $\omega^k$, by definition of $\omega^k$. In order to simplify the notation, define $$H^k\doteq Dg(x^k)[d^k]$$ and $$B^k_{\alpha\gamma}\doteq \tilde{H}_{\alpha\gamma}^k  \odot \mathcal{B}(\lambda^{S^k}(-g(x^k)))_{\alpha\gamma}.$$ Using the characterization of $V^k$ provided in \eqref{qi} from Proposition~\ref{sdp:propqi} (and Corollary~\ref{cor:qi}), we obtain

$$V^k[H^k]=S^k
\begin{bmatrix}
0 & 0 & 0 & B^k_{\alpha_+\gamma} \\
0 & 0 & 0 & B^k_{\alpha_-\gamma}\\
0 & 0 & V_{\vert\beta\vert}[\tilde{H}_{\beta\beta}] & \tilde{H}_{\beta\gamma}^k\\
(B^k_{\alpha_+\gamma})^\T  & (B^k_{\alpha_-\gamma})^\T & (\tilde{H}_{\beta\gamma}^k)^\T & \tilde{H}_{\gamma\gamma}^k \\
\end{bmatrix} (S^k)^\T. $$
Since $\rho_k \langle H^k,V^k[H^k]\rangle=\langle \tilde{H^k},\rho_k \tilde{V^k[H^k]}\rangle$, it is fundamental to note that for every $i\in \alpha_+$ and $j\in \gamma$,
\begin{equation}\label{sdp:conver-calb}
\left(\rho_k  \mathcal{B}(\lambda^{S^k}(-g(x^k)))\right)_{ij}=\frac{\rho_k  \lambda^{S^k}_j(-g(x^k))}{\lambda^{S^k}_j(-g(x^k))-\lambda^{S^k}_i(-g(x^k))}\to  \frac{\lambda^{U}_j(\omega\s)}{\lambda^{U}_i(g(x\s))},
\end{equation}
because $\rho_k \lambda^{S^k}_j(-g(x^k))=\lambda^{S^k}_j(\omega^k)$ and $\lambda^{S^k}_j(-g(x^k))\to 0$. Also, keep in mind that $$\lambda_i^{S^k}(-g(x^k))=-\lambda_i^{S^k}(g(x^k)).$$ The blocks indexed by $\alpha_-, \beta$, and $\gamma$, are all blocks of zeros because if $k$ is large enough, we must have $\vert\alpha_-\vert + \vert\beta\vert + \vert\gamma\vert=\bar{\beta}$ and, on the other hand, since $d^k\in S(x^k,x\s)$ we also have that $\tilde{H}_{\bar\beta\bar\beta}\doteq (S_{\bar\beta}^k)^\T Dg(x^k)[d^k]S_{\bar\beta}^k=0$. Similarly, $Dh(x^k)d^k=0$. Thus $$\lim_{k\to  \infty} \rho_k \langle H^k,V^k[H^k]\rangle=2\langle H,g(x\s)^{\dag} H \omega\s \rangle$$ and, consequently, \eqref{sdp:akkt2re} implies \eqref{sdp:wsoncre}, which means $x\s$ satisfies the WSOC with the multiplier $\omega\s$.  

\end{enumerate} 
\end{proof}

In the presence of nondegeneracy, the set of Lagrange multipliers is a singleton and Theorem \ref{sdp:minwsonc} recovers the classical result of~\cite{Shapiro1997}, but even without assuming uniqueness of the Lagrange multiplier it ensures there will be at least one multiplier satisfying  WSOC. Moreover, in contrast with~\cite{Jarre2012,Lourenco2018,Shapiro1997}, our proof does not require strict complementarity; but nevertheless, if it does hold, then the proof of Theorem~\ref{sdp:minwsonc} can be significantly simplified, since in this case the sequence $\{g(x^k)\}_{k\in \N}$ is nonsingular and we can avoid the use of subdifferentials.

\si{\section{Final remarks}\label{sec:conclusion}}
\jo{\section{Conclusion}\label{sec:conclusion}}

In this paper, we proved that every local minimizer of a nonlinear semidefinite program or a nonlinear second-order cone program satisfies the weak second-order necessary optimality condition under Robinson’s constraint qualification and the so-called weak constant rank property (WCR), which was extended from NLP \cite{ams07}. This joint condition is strictly weaker than nondegeneracy in NLP, NSOCP, and NSDP. We also stress that we do not assume strict complementarity, which is common in second-order analyses for conic programming. In contrast, our second-order necessary condition is based on the lineality space of the critical cone, and not the critical cone itself. This is consistent with the algorithmic practice of second-order algorithms as no algorithm is known to achieve a stronger second-order necessary optimality condition (see the extended version of \cite{Behling2017a} for details).

In the context of conic programming, several different approaches are known for obtaining second-order necessary optimality conditions \cite{bonnansramirez,Cominetti1990,Forsgren2000,Fukuda2016,Jarre2012,Lourenco2018,Shapiro1997}. We present a novel approach by extending the existing theory of first-order sequential optimality conditions to the second-order context. In particular, it is remarkable to see the appearance of the {\it sigma-term} in such a variety of approaches, which contributes to the understanding of this concept.

Our approach has a heavy algorithmic taste, as our proof is based on the construction of a sequence of approximate solutions of penalized subproblems, very similarly to a sequence generated by practical algorithms. In particular, a similar first-order approach has recently led to several improvements of global convergence theory of augmented Lagrangian methods in conic contexts \cite{Santos2019,Santos2021,Andreani2020,ahv}.

Thus, this paper opens the path to the development of second-order algorithms in conic optimization, which, as far as we know, has not been considered yet in the literature. In particular, augmented Lagrangian and interior point methods \cite{Birgin2016a,cakkt2} are expected to be well suited to the techniques we develop here. In this context, the joint condition ``Robinson’s CQ+WCR'' is the natural candidate for a condition to guarantee global convergence to a second-order stationary point.

\jo{
\section*{Statements and declarations}
The authors have no relevant financial or non-financial interests to disclose. Data sharing not applicable to this article as no datasets were generated or analysed during the current study.
}

\bibliographystyle{plain}


\end{document}